\documentclass[11pt]{amsart}
\usepackage{amsmath}
\usepackage{mathtools}
\usepackage{}
\usepackage{graphicx}
\usepackage[colorlinks=true, allcolors=blue]{hyperref}
\usepackage{amsfonts}
\usepackage{amsthm}
\usepackage{newlfont}
\usepackage{amscd}
\usepackage{amsgen}
\usepackage{amssymb}
\usepackage{mathrsfs}	
\usepackage{longtable}
\usepackage{listings}
\usepackage{extarrows}
\usepackage{tikz}
\usepackage{tikz-cd}
\usepackage{verbatim}
\numberwithin{equation}{section}
\usepackage[all]{xy}
\usepackage{color}
\usepackage{amssymb}
\usepackage[left=3cm,right=3cm]{geometry}
\usepackage{tikz}
\usepackage{tikz-cd}
\usepackage{mathtools}
\usepackage{bm} 
\usepackage{dynkin-diagrams}
\usetikzlibrary{matrix,shapes,arrows,decorations.pathmorphing}
\usepackage{calligra}
\usepackage{mathrsfs}
\usepackage{enumitem} 
\usepackage{tgpagella} 
\usepackage[parfill]{parskip} 
\usepackage{float}
\restylefloat{table}
\usepackage{ytableau}
\usepackage{adjustbox}
\usetikzlibrary{calc,matrix,positioning}

\tikzset{ 
    table/.style={
        matrix of nodes,
        row sep=-\pgflinewidth,
        column sep=-\pgflinewidth,
        nodes={rectangle, text width=2.5em, text height = 1.5em, align=center},
        text depth=1.25ex,
        text height=2.5ex,
        nodes in empty cells
    },
}

\let\blb\mathbb

\def\CC{{\blb C}}

\def\HH{{\blb H}}

\def\LL{{\blb L}}

\def\PP{{\blb P}}

\def\RR{{\blb R}}

\def\VV{{\blb V}}

\let\cal\mathcal
\def\Ac{{\cal A}}
\def\Bc{{\cal B}}

\def\Ec{{\cal E}}
\def\Fc{{\cal F}}
\def\Gc{{\cal G}}
\def\Hc{{\cal H}}

\def\Lc{{\cal L}}
\def\Mc{{\cal M}}

\def\Oc{{\cal O}}
\def\Pc{{\cal P}}
\def\Qc{{\cal Q}}

\def\Uc{{\cal U}}
\def\Vc{{\cal V}}
\def\Wc{{\cal W}}
\def\Xc{{\cal X}}

\def\Zc{{\cal Z}}

\def\Fm{{\mathfrak F}}

\DeclareMathOperator{\Sym}{Sym}

\newtheorem{lemma}{Lemma}[section]
\newtheorem{proposition}[lemma]{Proposition}
\newtheorem{theorem}[lemma]{Theorem}
\newtheorem{corollary}[lemma]{Corollary}
\newtheorem{condition}[lemma]{Condition}

\newtheorem{conjecture}[lemma]{Conjecture}

\newtheorem{todo}[lemma]{To do}

\theoremstyle{remark}

\newtheorem{remark}[lemma]{Remark}

\def\dbcoh{D^b}
\def\wt{\widetilde}

\def\arw{\longrightarrow}
\def\Hom{\operatorname{Hom}}

\def\Ext{\operatorname{Ext}}

\DeclareMathOperator{\sHom}{\mathscr{H}\text{\kern -3pt {\calligra\large om}}\,}

\newcommand\quotient[2]{
        \mathchoice
            {
                \text{\raise1ex\hbox{$#1$}\Big/\lower1ex\hbox{$#2$}}%
            }
            {
                #1\,/\,#2
            }
            {
                #1\,/\,#2
            }
            {
                #1\,/\,#2
            }
    }

\makeatletter
\def\namedlabel#1#2{\begingroup
    #2%
    \def\@currentlabel{#2}%
    \phantomsection\label{#1}\endgroup
}
\makeatother

\title{Fano fibrations and DK conjecture for relative Grassmann flips}

\author{Marco Rampazzo}
\address{
Alma Mater Studiorum Università di Bologna\\ Dipartimento di Matematica \\ Piazza di Porta San Donato 5\\ 40126 Bologna.}
\email[M.~ Rampazzo]{marco.rampazzo3@unibo.it}

\begin{document}

\begin{abstract}
    Given a vector bundle $\Ec$ on a smooth projective variety $B$, the flag bundle $\Fc l(1,2,\Ec)$ admits two projective bundle structures over the Grassmann bundles $\Gc r(1, \Ec)$ and $\Gc r(2, \Ec)$. The data of a general section of a suitably defined line bundle on $\Fc l(1,2,\Ec)$ defines two varieties: a cover $X_1$ of $B$, and a fibration $X_2$ on $B$ with general fiber isomorphic to a smooth Fano variety. We construct a semiorthogonal decomposition of the derived category of $X_2$ which consists of a list of exceptional objects and a subcategory equivalent to the derived category of $X_1$. As a byproduct, we obtain a new full exceptional collection for the Fano fourfold of degree $12$ and genus $7$. Any birational map of smooth projective varieties which is resolved by blowups with exceptional divisor $\Fc l(1, 2, \Ec)$ is an instance of a so-called Grassmann flip: we prove that the DK conjecture of Bondal--Orlov and Kawamata holds for such flips. This generalizes a previous result of Leung and Xie to a relative setting.
\end{abstract}

\maketitle

\section{Introduction}
    The role of the derived category of coherent sheaves as an invariant has been object of study for decades: the remarkable properties it exhibits in the context of Fano varieties \cite{bondalorlovreconstruction} and K3 surfaces \cite{orlov_k3} motivates the attempts to relate it to other geometric data, such as the class in the Grothendieck ring of varieties \cite{borisovcaldararu, kuznetsovshinder}, the Hodge structure and the isomorphism/birational class \cite{ottemrennemo,borisovcaldararuperry}. About the latter, several counterexamples rule out the possibility that this can be true in general, but it is conjectured that the derived category should behave like a birational invariant if we restrict our attention to a specific class of maps, called K-equivalences and K-inequalities. The conjecture, proposed by Bondal--Orlov \cite{bondalorlov} and Kawamata \cite{kawamata}, can be formulated as:\footnote{Note that several authors call ``DK-conjecture'' the statement about K-equivalences, without considering K-inquealities.}
    \begin{conjecture}[DK conjecture]\label{conj:dk_conjecture}
        Consider a K-inequality $\mu:\Xc_1\dashrightarrow\Xc_2$, i.e. a birational map resolved by two morphisms $g_i:\Xc\arw\Xc_i$ such that $g_2^*K_{\Xc_2}\sim_{\operatorname{lin}} g_1^*K_{\Xc_1} + D$, where $D$ is an effective divisor. Then, there is a fully faithful functor $\Phi: \dbcoh(\Xc_1)\xhookrightarrow{\,\,\,\,\,}\dbcoh(\Xc_2)$. Moreover, assume that $\mu$ is a K-equivalence, i.e. $D = 0$. Then, $\Phi$ is an equivalence of categories.
    \end{conjecture}
    Generalizing the work of Kanemitsu \cite{kanemitsu} on the so-called simple K-equivalences, Leung and Xie \cite{leung_xie_new} introduce the notion of \emph{simple flips} or \emph{simple K-inequalities}, i.e. K-inequalities as in Conjecture \ref{conj:dk_conjecture} such that $g_1$ and $g_2$ are blowups of smooth centers with the same exceptional divisor. For all such flips, the exceptional divisor is a family of special Fano varieties with two projective bundle structures which we call \emph{generalized roofs}: these varieties are classified in the homogeneous case \cite{ourpaper_generalizedroofs, leung_xie_new}. Flips associated to generalized homogeneous roofs are called \emph{generalized Grassmann flips}).\\
    Evidence for the DK conjecture, for generalized Grassmann flips (and flops) has been found in \cite{bondal_orlov_flips, kawamata, namikawa_mukai_flops, segal_C2, morimura_c2_a4, ueda_G2_flop, hara_uedaflop} for the flop case, and \cite{bondal_orlov_flips,leung_xie,leung_xie_new} for the flip case. All generalized Grassmannian flips which have been addressed (with $D\neq 0$) are such that the exceptional divisor is itself a generalized roof, i.e. the family is over a single point: in this paper, instead, we focus on the class of simple flips with exceptional divisor isomprphic to a partial flag bundle $\Fc l(1, 2, \Ec)$ over an arbitrary smooth base $B$, i.e. a locally trivial fibration with fiber isomorphic to $F(1,2,N)$ for some $N$. Our main result is the following:
    \begin{theorem}\label{main_theorem_DK_intro}
        Consider a Grassmann flip $\mu:\Xc_1\dashrightarrow\Xc_2$ such that the exceptional divisor of the flip is isomorphic to a flag bundle $\Fc l(1,2,\Ec)$ over a smooth projective base. Then $\dbcoh(\Xc_1)\subset\dbcoh(\Xc_2)$, i.e. $\mu$ satisfies the DK-conjecture.
    \end{theorem}
    This is a generalization of the main result of \cite{leung_xie} to the relative setting over a smooth base. As in the original work, the proof is based on constructing fully faithful embeddings of $\dbcoh(\Xc_1)$ and $\dbcoh(\Xc_2)$ in  $\dbcoh(\Xc)$, such that the semiorthogonal complements are generated by full exceptional collections of (pushforwards of) vector bundles on the exceptional divisor, and then proving the existence of an embedding of the semiorthogonal complements via mutations of exceptional objects. This is done with a diagrammatic technique called ``chess game'', which allows to easily visualize the mutations, otherwise very cumbersome to write. This kind of approach first appeared in \cite{kuznetsov_hpd, thomas_chess_game}, and then in the proof of the main results of \cite{leung_xie, leung_xie_new}: instead of directly generalizing the proof of \cite{leung_xie}, to the relative setting, we propose a modified and simplified version of the chess game. Besides the Grassmann flip, the geometry of $\Fc l(1, 2, \Ec)$ allows to introduce some interesting pairs of varieties, of which we describe the relation at the level of derived categories, by means of the same kind of chess game. Consider a general hyperplane section $\Mc\subset\Fc(1, 2, \Ec)$. Such variety has two contractions, with special fibers, respectively, over two smooth varieties $X_1$ and $X_2$. These are subvarieties of Grassmann bundles, and are zero loci of pushforwards of the section defining $\Mc$. Under the mild assumption \ref{cond:bpf}, by using the fact that the derived categories of $X_i$ admit fully faithful embeddings in $\dbcoh(\Mc)$, we prove the following:
    \begin{theorem}\label{thm:main_theorem_fibrations_intro}
        Let $X_1$ and $X_2$ be fibrations over $B$ as above. Then there is a fully faithful functor $\Psi:\dbcoh(X_1)\xhookrightarrow{\,\,\,\,\,}\dbcoh(X_2)$.
    \end{theorem}
    The proof of Theorem \ref{thm:main_theorem_fibrations_intro}, which is articulated in Proposition \ref{prop:main_theorem_fibrations_even} and Proposition \ref{prop:main_theorem_fibrations_odd}, allows to construct explicit semiorthogonal decompositions for $\dbcoh(X_2)$ containing $\dbcoh(X_1)$ as an admissible subcategory, and to describe the semiorthogonal complement in terms of pullbacks of $\dbcoh(B)$ twisted by appropriate objects. If we assume the further conditions \ref{cond:restriction_surjective} and \ref{cond:dimension_of_B}, the varieties $X_1$ and $X_2$ can be described as fibrations over $B$ where $X_1$ is a cover and the general fiber of $X_2$ is a smooth Fano variety.
    We provide an infinite series of examples of such embeddings, focusing on the cases where $B$, the $\Xc_i$'s and $\Xc$ are rational homogeneous varieties of type $A$, and the $X_i$'s are cut by general sections of homogeneous vector bundles. In particular, one of such examples provides an alternative full exceptional collection of length $16$ for the Fano fourfold of degree $12$ and genus $7$.
    \subsection*{Structure of the paper}
        In section \ref{sec:geoemtry_part_1} we describe the exceptional divisor of the Grassmann flip and its smooth hyperplane sections. We introduce the pairs of varieties $X_1, X_2$ of Theorem \ref{thm:main_theorem_fibrations_intro} and we discuss their properties. Then, the proof of Theorem \ref{thm:main_theorem_fibrations_intro} is explained in Section \ref{sec:derived_cats_cover_and_fibration}, with examples in Section \ref{sec:examples}. The Grassmann flip construction, along with the proof of Theorem \ref{main_theorem_DK_intro}, is addressed in Section \ref{sec:grassmann_flip}. Finally, we gather all Borel--Weil--Bott computations in Appendix \ref{app:borel_weil_bott}.

    \subsection*{Acknowledgments}
        This project originated from several discussions with Enrico Fatighenti, Micha\l\ Kapustka and Giovanni Mongardi. I would like to express my gratitude for all the suggestions, comments and insights they provided. I would like to thank Ying Xie for his valuable remarks on the first draft of this paper, and Naichung Conan Leung for inviting me to the Chinese University of Hong Kong to discuss about it. This work is supported by PRIN2020 ``2020KKWT53''. I am a member of GNSAGA of INDAM.
\section{Flag bundles and fibrations}\label{sec:geoemtry_part_1}
    \subsection{Notation}
        We shall work over the field of complex numbers. Given a general section $s$ of a vector bundle $E$ on a variety $X$, we call $Z(s)\subset X$ the zero locus of $s$. Moreover, for every $x\in X$, $E_x$ denotes the fiber of $E$ over the point $x$. We say a variety $Y$ has a projective bundle structure over $X$ if there is an isomorphism $Y\simeq \PP(E)$ for some vector bundle $E$ on some variety $X$, sometimes the projective bundle will be denoted by $\PP(E\arw X)$. We use Grothendieck's convention for projectivizations: given $\pi:\PP(E)\arw X$, there is a line bundle $L$ on $\PP(E)$ such that $L|_{\pi^{-1}(x)}\simeq\Oc_{\pi^{-1}(x)}(1)$ and such that $\pi_* L\simeq E$. We denote by $G(k, V_n)$ the Grassmannian parametrizing $k$-linear spaces in a fixed vector space $V_n\simeq\CC^n$. To unburden the notation, we sometimes use the expression $G(k, n)$. Similarly, we call $F(k_1, \dots, k_m, V_n)$ the (partial) flag variety parametrizing $m$-tuples $(V_{k_1}, \dots, V_{k_m})$ such that $V_{k_1}\subset V_{k_2}\subset \dots \subset V_{k_m}\subset V_n$. To denote cohomology of vector bundles, we write a direct sum of vector spaces shifted by the negative of the appropriate degree, e.g. $H^\bullet(\PP^n, \Oc\oplus\Oc(-n-1)) \simeq \CC[0]\oplus\CC[-n]$. The bounded derived category of coherent sheaves of a smooth projective variety $X$ will be indicated as $\dbcoh(X)$.
    \subsection{The flag bundle}
        Consider a vector bundle $\Ec\arw B$ of rank $2n+\epsilon$ over a smooth, projective base $B$, for $n\geq 2$ and $\epsilon\in\{0,1\}$, where the choice of notation is due to the fact that the proofs of Propositions \ref{prop:main_theorem_fibrations_even} and \ref{prop:main_theorem_fibrations_odd} depend on the parity of the rank of $B$. Then, call $\Fc l(1, 2, \Ec)$ the associated flag bundle over $B$, i.e. the locally trivial fibration with fiber $F(1,2,\Ec_b)$ for every $b\in B$. Similarly, the associated Grassmann bundle of $k$-subspaces in $\Ec_b$ for every $b$ will be called $\Gc r(k, \Ec)$ (in particular, one has $\Gc r(1, \Ec) = \PP(\Ec)$). We have the following commutative diagram, where all maps are locally trivial:
        \begin{equation}\label{eq:main_diagram}
            \begin{tikzcd}[row sep = large, column sep = tiny]
                & \Fc l (1,2, \Ec)\ar[swap]{dl}{p_1}\ar{dr}{p_2} &  \\
                \Gc r(1, \Ec)\ar[swap]{dr}{r_1} & & \Gc r(2, \Ec)\ar{dl}{r_2} \\
                &B& 
            \end{tikzcd}
        \end{equation}
        Recall that $\Gc r(k, \Ec)$ comes with a relative tautological short exact sequence:
        \begin{equation}\label{eq:relative_euler_grassmann_bundle}
            0\arw \Uc_k\arw r_k^*\Ec\arw \Qc_k\arw 0
        \end{equation}
        Now, consider the line bundle $\Lc:= p_1^*\Uc_1^\vee\otimes p_2^*\wedge^2\Uc_2^\vee$. Then, one has $p_{1*}\Lc \simeq \Uc_1^\vee\otimes\wedge^{2n+\epsilon-2}\Qc_1$, and $p_{2*}\Lc\simeq \Uc_2^\vee\otimes\wedge^2\Uc_2^\vee$. Moreover, $\Fc l(1,2,\Ec) \simeq\PP(p_{1*}\Lc\arw\Gc r(1, \Ec))\simeq\PP(p_{2*}\Lc\arw \Gc r(2, \Ec))$.\\
        One has:
        \begin{equation*}
            \begin{split}
                \omega_{\Gc r(k, \Ec)} & = \det(\Uc_k\otimes\Qc_k^\vee) \\
                & = \det(\Uc_k)^{\otimes(2n+\epsilon-k)}\otimes \det(\Qc_k^\vee)^{\otimes k}\\ & = \det(\Uc_k)^{\otimes 2n+\epsilon}\otimes \det(\Ec^\vee)^{\otimes k}.
            \end{split}
        \end{equation*}
        Moreover, by combining the relative Euler sequence and the relative tangent bundle sequence associated to $p_k$, we find:
        \begin{equation*}
            \omega_{\Fc l(1, 2, \Ec)} = \Uc_1^{\otimes 2}\otimes\wedge^2\Uc_2^{\otimes(2n+\epsilon-1)}\otimes\det\Ec^{\otimes (2n+\epsilon-3)}.
        \end{equation*}
        \begin{remark}
            The fibration $\Fc l(1, 2, \Ec)\arw B$ is an example of a family of \emph{generalized homogeneous roofs}, i.e. rational homogeneous varieties of Picard rank two which admit two projective bundle structures. These objects are generalizations of the homogeneous roofs introduced in \cite{kanemitsu}: they have been classified in \cite{ourpaper_generalizedroofs} in the context of derived categories of Fano varieites and, independently, in \cite{leung_xie_new} in the context of flips, while a classification for non-homogeneous cases has yet to be found.
        \end{remark}
        \subsection{Two fibrations over \texorpdfstring{$B$}{something}}\label{sec:cover_and_fano_fibration}
        Consider now a general section $s\in H^0(\Fc l(1,2,\Ec), \Lc)$ and call $\Mc\in\Fc l(1,2,\Ec)$ its zero locus. Define $X_k:= Z(p_{k*}s)$ for $k = 1,2$. By adjunction, one finds:
        \begin{equation*}
            \begin{split}
                \omega_\Mc & = \Uc_1\otimes\wedge^2\Uc_2^{\otimes(2n+\epsilon-2)}\otimes\det\Ec^{\otimes(2n+\epsilon-3)} \\
                \omega_{X_1} & = (\Uc_1^\vee)^{\otimes(2n+\epsilon-3)}\otimes\det(\Ec)^{\otimes(2n+\epsilon-3)} \\
                \omega_{X_2} & = (\wedge^2\Uc_2^\vee)^{\otimes(3-2n-\epsilon)}\otimes\det(\Ec^\vee)^{\otimes 2}.
            \end{split}
        \end{equation*}
        In general, there is no guarantee that the general section of $\Lc$ cuts a smooth variety, or even that $\Lc$ has global sections. However, $\Lc$ is defined up to pullbacks of line bundles from $B$, and therefore we have some freedom to choose it so that it has global sections (this choice, in light of Equation \ref{eq:relative_euler_grassmann_bundle}, is a consequence of a choice of a twist of $\Ec$ by a line bundle on $B$). By restricting $r_1$ and $r_2$, the varieties $X_1$ and $X_2$ inherit fibration structures over $B$. To achieve some control on the fibers, we introduce the following conditions on our setup.
        \begin{condition}\label{cond:bpf}
            The bundle $\Lc$ is basepoint free.
        \end{condition}
        \begin{condition}\label{cond:restriction_surjective}
            The restriction map
            \begin{equation*}
                \rho_b: H^0(\Fc l(1,2,\Ec), \Lc) \arw H^0(F(1,2,\Ec_b), \Oc(1,1))
            \end{equation*}
            is surjective.
        \end{condition}
        \begin{condition}\label{cond:dimension_of_B}
            The dimension of $B$ is smaller than $\binom{2n+\epsilon+1}{2} - 2n - \epsilon$.
        \end{condition}
        In fact, assuming that Conditions \ref{cond:bpf}, \ref{cond:restriction_surjective} and \ref{cond:dimension_of_B} hold, one has:
        \begin{lemma}\label{lem:the_cover_of_B}
            $X_1$ is a smooth $N : 1$ cover of $B$, where:
            \begin{equation}
                N = \sum_{i = 0}^{2n+\epsilon-1}(-3)^{2n+\epsilon-i-1}2^i\binom{2n+\epsilon}{i}.
            \end{equation}
            $X_2$ is a smooth fibration over $B$, with general fiber isomorphic to a Fano variety of codimension two and coindex three in $G(2, 2n+\epsilon)$.
        \end{lemma}
        \begin{proof}
            Take $\gamma_i:=r_i|_{X_i}$. Given any $b\in B$, the expected codimension of the fibers $\gamma_i^{-1}(b)$, being zero loci of the restrictions $\Uc_1^\vee\otimes\wedge^{2n+\epsilon-2}\Qc_1|_{r_1^{-1}(b)} \simeq \Qc_{G(1, \Ec_b)}^\vee(2)$ and  $\Uc_2^\vee\otimes\wedge^2\Uc_2^\vee|_{r_2^{-1}(b)} \simeq \Uc_{G(2, \Ec_b)}^\vee(1)$ is clear, along with the coindex of the general fiber of $\gamma_2$. In particular, by Condition \ref{cond:bpf} the generality of $s$ implies that $\Mc$, $X_1$, $X_2$ are smooth, and by Condition \ref{cond:restriction_surjective} we impose that the general fiber of $\gamma_2$ is smooth of expected dimension, and the general fiber of $\gamma_1$ is a set of $N$ distinct points. To determine $N$, we just need to compute the degree of the top Chern class of $\Qc_{G(1, \Ec_b)}^\vee(2)$. Twisting the dual tautological sequence of $G(1, \Ec_b)$ one has
            \begin{equation}\label{eq:twisted_tautological_sequence_G1n}
                0\arw \Qc_{G(1, \Ec_b)}^\vee(2)\arw \Ec_b\otimes\Oc(2) \arw \Oc(3)\arw 0,
            \end{equation}
            and therefore, if we call $H$ the hyperplane class on $G(1, \Ec_b)$:
            \begin{equation*}
                \begin{split}
                    c(\Qc_{G(1, \Ec_b)}^\vee(2)) & = (1+2H)^{2n+\epsilon}/(1+3H)\\
                    & = \sum_{i=0}^{2n+\epsilon}\binom{2n+\epsilon}{i}(2H)^i\sum_{j=0}^{2n+\epsilon-1}(-3H)^j
                \end{split}
            \end{equation*}
            where the polynomial is truncated at degree $2n +\epsilon -1$ because $H^{2n+\epsilon} = 0$. The degree of the top Chern class is precisely the coefficient of the term of maximal degree, and it can be easily computed to be exactly $N$. Let us now rule out the existence of positive dimensional fibers for $\gamma_1$. Call $\HH$ the space of global sections of $\Qc_{G(1, \Ec_b)}^\vee(2)$. First, consider the variety:
            \begin{equation*}
                \Vc = \{ (b, \sigma)\in B\times \HH : \dim(Z(\sigma_b))>0\}.
            \end{equation*}
            Then, one has the obvious projections
            \begin{equation*}
                \begin{tikzcd}
                    & \Vc \ar[swap]{dl}{pr_1}\ar{dr}{pr_2} & \\
                    B & & \HH.
                \end{tikzcd}
            \end{equation*}
            Hence, by computing the dimension of $\Vc$, one easily sees that for every $b\in B$ the codimension of the space of sections of $\Qc_{G(1, \Ec_b)}^\vee(2)$ with positive dimensional zero loci is equal to the codimension in $B$ of the subset over which $\gamma_1$ has positive dimensional fibers. We can read from the sequence \ref{eq:twisted_tautological_sequence_G1n} that sections of $\Qc_{G(1, \Ec_b)}^\vee(2)$ are elements of the kernel of
            \begin{equation*}
                \begin{tikzcd}[row sep = tiny, column sep = large, /tikz/column 1/.append style={anchor=base east} ,/tikz/column 2/.append style={anchor=base west}]
                    f: H^0(G(1, \Ec_b), \Ec_b\otimes\Oc(2)) \ar[two heads]{r} & H^0(G(1, \Ec_b), \Oc(3))\\
                    (q_1, \dots, q_{2n+\epsilon}) \ar[maps to]{r} & x_1 q_1+\dots + x_{2n+\epsilon}q_{2n+\epsilon}
                \end{tikzcd}
            \end{equation*}
            where $(x_1, \dots, x_{2n+\epsilon})$ are linear maps on the coordinates. A higher dimensional zero locus appears only if the quadrics $q_i$ are either linearly dependent, or are annihilated by a degree one syzygy. The first case is represented by sections lying in the kernel of a morphism
            \begin{equation*}
                (f, \lambda I): H^0(G(1, \Ec_b), \Ec_b\otimes\Oc(2)) \arw H^0(G(1, \Ec_b), \Oc(3)\oplus\Oc(2))
            \end{equation*}
            where $f$ is the contraction with a vector of linear entries, and $\lambda I$ is a nonzero scalar multiple of the identity. The dimension of the space of such sections, therefore, will be equal to the dimension of the kernel of $(f, \lambda I)$ plus the dimension of the space of functions $(f, \lambda I)$. We obtain:
            \begin{equation}\label{eq:number_sections_LD}
                H^0(G(1, \Ec_b), \Qc_{G(1, \Ec_b)}^\vee(2)) - \binom{2n+\epsilon+1}{2} + 2n + \epsilon
            \end{equation}
            On the other hand, sections which share a second linear syzygy are elements of the kernel of
            \begin{equation*}
                (f, f'): H^0(G(1, \Ec_b), \Ec_b\otimes\Oc(2)) \arw H^0(G(1, \Ec_b), \Oc(3)\oplus\Oc(3))
            \end{equation*}
            where $f, f'$ are both contractions with a vector of linear entries. Here the dimension count gives
            \begin{equation}\label{eq:number_sections_with_syzygy}
                H^0(G(1, \Ec_b), \Qc_{G(1, \Ec_b)}^\vee(2)) - \binom{2n+\epsilon+2}{3} + 4n + 2\epsilon -1
            \end{equation}
            By comparing the parameter counts \ref{eq:number_sections_LD} and \ref{eq:number_sections_with_syzygy}, we conclude that for $\dim B < \binom{2n+\epsilon+1}{2} - 2n - \epsilon$, the fibration $\gamma_1$ cannot have positive dimensional fibers, but this is exactly Condition \ref{cond:dimension_of_B}, which we assume to hold.
        \end{proof}
\section{A semiorthogonal decomposition for \texorpdfstring{$X_2$}{something}}\label{sec:derived_cats_cover_and_fibration}
    In this section we produce a semiorthogonal decomposition for $X_2$ containing the derived category of $X_1$. This is done by constructing two semiorthogonal decompositions for $\Mc$ containing, respectively, $\dbcoh(X_1)$ and $\dbcoh(X_2)$ as admissible subcategories, and then mutating the semiorthogonal complements until one has $\dbcoh(X_2)^\perp \subset \dbcoh(X_1)^\perp$. To visualize the mutations, we use a modified version of the ``chess game'' of \cite{thomas_chess_game, leung_xie}.
    \subsection{Two semiorthogonal decompositions for \texorpdfstring{$\Mc$}{something}}\label{sec:sods_for_M}
        Consider the morphism $q_1 :\Mc\arw \Gc r(1, \Ec)$ of relative dimension (generally) $2n+\epsilon-2$, obtained by restricting $p_1$ to $\Mc$. This is an instance of the so-called ``Cayley trick''. Consider the following diagram:
        \begin{equation*}
            \begin{tikzcd}[row sep = huge, column sep = huge]
                E_1\ar{d}{\bar q_1}\ar[hook]{r}{i_1} & \Mc
                \ar{d}{q_1} \\
                X_1 \ar[hook]{r} & \Gc r(1, \Ec)
            \end{tikzcd}
        \end{equation*}
        where $\bar q_1$ is the base change of $q_1$ to $X_1$, and it is the projectivization of the normal bundle of $X_1$ in $\Gc r(1, \Ec)$.
        By \cite{cayleytrick}, we have a semiorthogonal decomposition
        \begin{equation*}
            \dbcoh(\Mc) = \langle i_{1*}\bar q^* \dbcoh(X_1), q_1^* \dbcoh(\Gc r(1, \Ec))\otimes\Lc, \dots, \dbcoh(\Gc r(1, \Ec))\otimes\Lc^{\otimes(2n+\epsilon-2)} \rangle.
        \end{equation*}
        Now, since $\Gc r(1, \Ec) \arw B$ is a projective bundle with fiber $\PP^{2n+\epsilon -1}$, by \cite[Section 2]{orlovblowup} we can write $\dbcoh(\Gc r(1, \Ec))$ as follows:
        \begin{equation}\label{eq:sod_Gr1E}
            \dbcoh(\Gc r(1, \Ec)) = \langle r_1^*\dbcoh(B), r_1^*\dbcoh(B)\otimes\Uc_1^\vee, \dots, r_1^*\dbcoh(B)\otimes(\Uc_1^\vee)^{\otimes(2n+\epsilon-1)}\rangle,
        \end{equation}
        where we used the fact that the list of $r_1$-relatively exceptional objects $\{\Oc, \Uc_1^\vee, \dots, (\Uc_1^\vee)^{\otimes(2n+\epsilon -1)}\}$ restricts to each fiber to the full exceptional collection $\langle\Oc, \Oc(1) \dots, \Oc(2n+\epsilon-1)\rangle$ of $\dbcoh(\PP^{2n+\epsilon-1})$.
        Summing all up, we have:
        \begin{equation}\label{eq:sod_for_M_with_line_bundles}
            \begin{split}
                \dbcoh(\Mc) = \langle & i_{1*}\bar q_1\dbcoh(X_1),\\
                & r_1^*\dbcoh(B)\otimes\Lc, \dots, r_1^*\dbcoh(B)\otimes\Lc\otimes (\Uc_1^\vee)^{\otimes (2n+\epsilon-1)}\\
                & r_1^*\dbcoh(B)\otimes\Lc^{\otimes 2}, \dots, r_1^*\dbcoh(B)\otimes\Lc^{\otimes 2}\otimes (\Uc_1^\vee)^{\otimes (2n+\epsilon-1)}\\
                & \vdots \\
                & r_1^*\dbcoh(B)\otimes\Lc^{\otimes (2n+\epsilon-2)}, \dots, r_1^*\dbcoh(B)\otimes\Lc^{\otimes (2n+\epsilon-2)}\otimes (\Uc_1^\vee)^{\otimes (2n+\epsilon-1)}\rangle
            \end{split}
        \end{equation}
        We can perform the same operations with $\bar p_2: M\arw \Zc_2$. First we use the result of \cite{cayleytrick} to write
        \begin{equation}\label{eq:decomposition_Mc_from_X2}
            \dbcoh(\Mc) = \langle i_{2*}\bar q_2^* \dbcoh(X_2), q_2^* \dbcoh(\Gc r(2, \Ec))\otimes\Lc \rangle.
        \end{equation}
        Then, recall that one has the following semiorthogonal decompositions for $\dbcoh(\Gc r(2, \Ec))$, due to \cite{kuznetsov_isotropic_lines, samokhin}:
        \begin{equation} \label{eq:sod_Gr2E}
            \dbcoh(\Gc r(2, \Ec)) = \left\{
                \begin{array}{cc}
                    \langle \Bc, \dots, \Bc\otimes(\wedge^2\Uc_2^\vee)^{\otimes (n - 2)}, \Ac\otimes(\wedge^2\Uc_2^\vee)^{\otimes (n-1)}, \dots, \Ac\otimes(\wedge^2\Uc_2^\vee)^{\otimes (2n-1)}\rangle & \epsilon = 0 \\
                    \langle \Bc, \dots, \Bc\otimes(\wedge^2\Uc_2^\vee)^{\otimes 2n} \rangle & \epsilon = 1
                \end{array}
            \right.
        \end{equation}
        where $\Ac = \{\Oc, \Uc_2^\vee, \dots, \Sym^{n-2}\Uc_2^\vee\}$ and $\Bc = \{\Ac, \Sym^{n-1}\Uc_2^\vee\}$.
        Summing all up, we will find the following semiorthogonal decomposition for $\epsilon = 1$:
        \begin{equation*}
            \dbcoh(\Mc) = \langle i_{2*}\bar q_2^* \dbcoh(X_2), q_2^*\Bc\otimes\Lc, \dots, q_2^*\Bc\otimes(\wedge^2\Uc_2^\vee)^{\otimes (n-1)}\otimes\Lc \rangle,
        \end{equation*}
        and the following for $\epsilon = 0$:
        \begin{equation*}
            \begin{split}
                \dbcoh(\Mc) = \langle & i_{2*}\bar q_2^* \dbcoh(X_2), q_2^*\Bc\otimes\Lc, \dots, q_2^*\Bc\otimes(\wedge^2\Uc_2^\vee)^{\otimes (n/2 -1)}\otimes\Lc,\\
                & q_2^*\Ac\otimes(\wedge^2\Uc_2^\vee)^{\otimes (n/2)}\otimes\Lc, \dots, q_2^*\Ac\otimes(\wedge^2\Uc_2^\vee)^{\otimes (n-1)}\otimes\Lc \rangle.
            \end{split}
        \end{equation*}
        Thus, we are ready to compare the semiorthogonal complements of $\dbcoh(X_1)$ and $\dbcoh(X_2)$ inside $\dbcoh(\Mc)$. However, the number of components grows wildly with $n$, and writing a sequence of mutations easily becomes a cumbersome task. For this reason, in the next pages (Sections \ref{sec:extending_bundles_from_fiber}, \ref{sec:the_chess_game_rules}) we introduce a ``chess game'' representation of such lists of objects, i.e. a diagrammatic language which allows to visualize mutations in a simple way. The chess game (including its name) is inspired by the works \cite{kuznetsov_hpd, thomas_chess_game, leung_xie}, but with different rules and symbols.
    \subsection{Extending bundles from the fibers}\label{sec:extending_bundles_from_fiber}
        Before introducing our version of the chess game, we need some technical results about mutations of subcategories of $\dbcoh(\Mc)$. We begin by recalling a different description of the flag bundle. Consider a principal $G$-bundle $\Vc\arw B$, where $G= SL(2n+\epsilon)$. Call $P_{1,2}\subset G$ the parabolic subgroup given by the elements of this form:
        \begin{equation*}
            p = \left(
            \begin{array}{ccc}
                \lambda_1 & \times & \times \\
                0 & \lambda_2 & \times \\
                0 & 0 & h
            \end{array}
            \right) \in SL(2n+\epsilon)
        \end{equation*}
        where $\lambda_1, \lambda_2\in \CC^*$, $h\in GL(2n+\epsilon)$ and the $\times$'s denote submatrices on which we impose no condition. One has $F(1, 2, 2n+\epsilon) = G/P_{1,2}$. Then, we can construct a locally trivial fibration $\Fc\arw B$ with fiber $F(1, 2, 2n+\epsilon)$ (i.e. our flag bundle) by taking $\Fc = \Vc\times^G G/P_{1,2}$, where the notation $\times^G$ denotes the quotient of the product by the equivalence relation $(g.t, v)\simeq (t, g.v)$ for all $g\in G$ and $(t, v)\in \Vc\times G/P_{1,2}$. Note that the choice $G= SL(2n+\epsilon)$ is purely motivated by writing the explicit description of the parabolic subgroup of a flag variety of type $A$: on the other hand, the construction of homogeneous vector bundles, together with their extension to generalized flag bundles, works for any choice of $G/P$.\\
        \\
        Let us call $\pi:\Fc\arw B$ the map induced by the structure map $\Vc\arw B$. Then, for every $b\in B$ we have $\pi^{-1}(b)\simeq F(1,2,2n+\epsilon)$. Recall that, given a principal $H$-bundle $\Wc$ over a variety $X$, there is the following exact functor from the category of $H$-modules to the category of vector bundles over $X$ (see \cite[Section 2.2]{nori}, or the survey \cite[Page 8]{balaji_newstad_notes}):
        \begin{equation*}
            \begin{tikzcd}[column sep = huge]
                H-\operatorname{Mod}\ar{r}{\Wc\times^H (-)} &  \operatorname{Vect}(X)
            \end{tikzcd}
        \end{equation*}
        which sends the $H$-module $R$ to the vector bundle $\Wc\times^H R$. In our case, $\Fc = \Vc\times^G G$ is a principal $P_{1,2}$-bundle over $\Zc$, because $\Vc\arw\Vc/P_{1,2}$ is a principal $P_{1,2}$-bundle and $\Vc/P_{1,2}\simeq \Vc\times^G G/P_{1,2}\simeq\Fc$ (see, for instance, \cite[Proposition 3.5]{mitchell_notes_principal_bundles}). This allows to construct an exact functor:
        \begin{equation*}
            \begin{tikzcd}[column sep = huge]
                P_{1,2}-\operatorname{Mod}\ar{r}{\Vc\times^G G\times^{P_{1,2}} (-)} &  \operatorname{Vect}(\Fc).
            \end{tikzcd}
        \end{equation*}
        Moreover, it is well-known that there exists an equivalence of categories
        \begin{equation*}
            \begin{tikzcd}[column sep = huge]
                P_{1,2}-\operatorname{Mod}\ar{r}{\VV_{G,P_{1,2}}} & \operatorname{Vect}^P(G/P_{1,2})
            \end{tikzcd}
        \end{equation*}
        which sends a $P_{1,2}$-module $H$ to the $P_{1,2}$-homogeneous vector bundle $\VV_{G,P_{1,2}}(H) = G\times^{P_{1,2}} H$. In particular, $\VV_{G,P_{1,2}}^{-1}$ is an exact functor. Summing all up we can construct an exact functor $\Fm$ sending homogeneous vector bundles over $G/P_{1,2} = F(1,2,2n+\epsilon)$ to vector bundles over $\Fc$:
        \begin{equation}\label{eq_relativizationfunctor}
            \begin{tikzcd}[column sep = huge, row sep = large]
                \operatorname{Vect}^{P_{1,2}}(G/P) \ar{rr}{\Fm:=\Vc\times^G G\times^{P_{1,2}} (-) \circ \VV_{G,P_{1,2}}^{-1}}\ar[swap]{dr}{\VV_{G,P_{1,2}}^{-1}} & & \operatorname{Vect}(\Fc)\\
                & P_{1,2}-\operatorname{Mod}\ar[swap]{ur}{\Vc\times^G G\times^{P_{1,2}} (-)} &
            \end{tikzcd}
        \end{equation}
        In particular, we have $\Oc(xh_1+yh_2) = \Fm(\Oc(x, y)$ for all $x, y$. Moreover, since $\Uc^\vee$ is homogeneous (although not irreducible), it is easy to see that $\Sym^m\Uc^\vee(x, y) = \Fm(\Sym^m\Uc_{G(2, 2n+\epsilon)}^\vee(x, y))$ for all $x, y$.
    \subsection{The chess game - rules}\label{sec:the_chess_game_rules}
        \subsubsection{Mutations of blocks}\label{sec:arrangement_of_boxes}
            Let us fix the following notation, where we omit pullbacks: $\Oc(xh_1 + yh_2):= (\Uc_1^\vee)^{\otimes x}\otimes(\wedge^2\Uc_2^\vee)^{\otimes y}$. In particular, one has $\Lc = \Oc(h_1+h_2)$. Let us also introduce the subcategories:
            \begin{equation*}
                \begin{split}
                    \mathbf S^m_{x,y}:= & \Sym^m \Uc_2^\vee(xh_1+yh_2)\otimes q_1^*r_1\dbcoh(B) \\
                    \mathbf A^m_{x,y}:= & \langle \mathbf S^0_{x,y}, \mathbf S^1_{x,y}, \dots, \mathbf S^m_{x,y}\rangle.
                \end{split}
            \end{equation*}
            In this language, we enunciate the following technical lemma:
            \begin{lemma}\label{lem:chess_game_exts}
                For $0\leq r\leq\min(t-2, n-1)$ and $r\neq t-1$ one has:
                \begin{equation*}
                    \LL_{\mathbf S^0_{t, 0}}\mathbf S^r_{-1, 1} \simeq \mathbf S^r_{-1, 1}.
                \end{equation*}
                Moreover, for $r = t-1$:
                \begin{equation*}
                    \LL_{\mathbf S^0_{t, 0}}\mathbf S^{t-1}_{-1, 1} \simeq \mathbf S^t_{0,0}.
                \end{equation*}
            \end{lemma}
            \begin{proof}
                Let us start by considering the following adjoint pair of functors:
                \begin{equation*}
                    \begin{split}
                        f: & E\longmapsto \rho^*E\otimes\iota^*\Oc(th_1)\\
                        f^! : & R \longmapsto \rho_* R\Hc om_\Mc(\iota^*\Oc(th_1), R)
                    \end{split}
                \end{equation*}
                where $f$ is a fully faithful embedding of $\dbcoh(B)$ in $\dbcoh(\Mc)$), the ``$!$'' symbol denotes the right adjoint, $\iota:\Mc\arw\Fc l(1,2,2n+\epsilon)$ is the embedding as a hypersurface and $\rho :=r_i\circ p_i|_{\Mc}$. Now, any object in $\mathbf S^r_{-1, 1}$ has the form $\iota^*\Sym^r\Uc_2^\vee(-h_1+h_2)\otimes\rho^* E$ for some $E\in\dbcoh(B)$, and its mutation through $\mathbf S^0_{t, 0}$is defined by the distinguished triangle
                \begin{equation}\label{eq:mutation_distinguished_triangle}
                    \begin{split}
                        ff^!(\iota^*\Sym^r\Uc_2^\vee(-h_1+h_2)\otimes\rho^* E) \arw \iota^*\Sym^r\Uc_2^\vee(-h_1+h_2)\otimes\rho^* E \arw \\
                        \arw \LL_{\mathbf S^0_{t, 0}}(\iota^*\Sym^r\Uc_2^\vee(-h_1+h_2)).
                    \end{split}
                \end{equation}
                Then, by adjunction:
                \begin{equation}\label{eq:mutations_intermediate_step}
                    \begin{split}
                        f^!(\iota^*\Sym^r\Uc_2^\vee(-h_1+h_2)\otimes\rho^* E) \simeq \hspace{150pt}\\
                        \simeq \rho_* R\Hc om_\Mc(\iota^*\Oc(th_1), \iota^*\Sym^r\Uc_2^\vee(-h_1+h_2)\otimes\rho^* E)\\
                        \simeq \rho_* \iota^* R\Hc om_{\Fc l(1, 2, \Ec)}(\Oc(th_1), \iota_*\iota^*\Sym^r\Uc_2^\vee(-h_1+h_2)\otimes\rho^* E)\\
                        \simeq r_{i*}p_{i*}\iota_* \iota^* R\Hc om_{\Fc l(1, 2, \Ec)}(\Oc(th_1), \iota_*\iota^*\Sym^r\Uc_2^\vee(-h_1+h_2)\otimes\rho^* E).
                    \end{split}
                \end{equation}
                Let us focus our attention on the term $R\Hc om_{\Fc l(1, 2, \Ec)}(\Oc(th_1), \iota_*\iota^*\Sym^r\Uc_2^\vee(-h_1+h_2)\otimes\rho^* E)$. For any $b\in B$ one has
                \begin{equation*}
                    \begin{split}
                        (r_{i*}p_{i*}R\Hc om_{\Fc l(1, 2, \Ec)}(\Oc(th_1), \iota_*\iota^*\Sym^r\Uc_2^\vee(-h_1+h_2)\otimes\rho^* E))_b \simeq \\
                        \simeq H^\bullet((r_i\circ p_i)^{-1}(b), \iota_*\iota^*\Sym^r\Uc_2^\vee(-(1+t)h_1+h_2)\otimes\rho^* E|_{(r_i\circ p_i)^{-1}(b)})\\
                        \simeq H^\bullet(F(1, 2, 2n+\epsilon), i_*i^*\Sym^r\Uc_{G(2, 2n+\epsilon)}^\vee(-1-t, 1))\\
                        \simeq \Hom_M^\bullet(\Oc(t, 0), \Sym^r\Uc_{G(2, 2n+\epsilon)}^\vee(-1, 1))
                    \end{split}
                \end{equation*}
                where $i$ is the embedding of $M$ in $F(1, 2, 2n+\epsilon)$. If $r\neq t-1$ the latter is zero by Lemma \ref{lem:main_coh_lemma}, and hence $R\Hc om_{\Fc l(1, 2, \Ec)}(\Oc(th_1), \iota_*\iota^*\Sym^r\Uc_2^\vee(-h_1+h_2)\otimes\rho^* E)$ vanishes identically. Therefore, the third term in the triangle \ref{eq:mutation_distinguished_triangle} is the cone over the zero map, and hence it will be isomorphic to the second one, proving the first part of the claim. Let us now consider the case $r = t-1$. Back to Equation \ref{eq:mutations_intermediate_step}, we can further manipulate such expression obtaining:
                \begin{equation*}
                    \begin{split}
                        f^!(\iota^*\Sym^r\Uc_2^\vee(-h_1+h_2)\otimes\rho^* E) & \simeq r_{i*}p_{i*} \iota_*\iota^* \Sym^r\Uc_2^\vee(-(1+t)h_1+h_2)\otimes\rho^* E)\\
                    \end{split}
                \end{equation*}
                Let us resolve $\iota_*\iota^*\Sym^r\Uc_2^\vee(-h_1+h_2)\otimes\rho^* E$ by the Koszul resolution of $\Mc$ (up to twists by pullbacks from $B$):
                \begin{equation*}
                    \begin{split}
                        0\arw \Sym^r\Uc_2^\vee(-(2+t)h_1)\otimes\rho^* E \arw  \Sym^r\Uc_2^\vee(-(1+t)h_1+h_2)\otimes\rho^* E)\arw \\
                        \arw \iota_*\iota^* \Sym^r\Uc_2^\vee(-(1+t)h_1+h_2)\otimes\rho^* E \arw 0
                    \end{split}
                \end{equation*}
                Here, by \ref{lem:main_coh_lemma} and the same argument as above, we have for all $b$:
                \begin{equation*}
                    \begin{split}
                        (r_{i*}p_{i*} \Sym^{t-1}\Uc_2^\vee(-(2+t)h_1)\otimes\rho^* E)_b & = 0\\
                        (r_{i*}p_{i*} \Sym^{t-1}\Uc_2^\vee(-(1+t)h_1+h_2)\otimes\rho^* E)_b & \simeq \CC[-1]
                    \end{split}
                \end{equation*}
                which tells us that $r_{i*}p_{i*} \iota_*\iota^* \Sym^r\Uc_2^\vee(-(1+t)h_1+h_2)\otimes\rho^* E)$ is a line bundle on $B$ shifted by $-1$. Hence, up to twists by pullbacks of line bundles on $B$, one has $ff^!(\iota^*\Sym^r\Uc_2^\vee(-h_1+h_2)\otimes\rho^* E) \simeq  \rho^* E\otimes\Oc(t h_1)$, and therefore the triangle \ref{eq:mutation_distinguished_triangle} reduces to a short exact sequence where $\LL_{\mathbf S^0_{t, 0}}(\iota^*\Sym^r\Uc_2^\vee(-h_1+h_2))$ is the extension. Such sequence is a twist of \ref{eq:symresolution_flag}, and hence we conclude the proof of the claim.
            \end{proof}
            \begin{lemma}\label{lem:chess_game_exts_II}
                For $0\leq m\leq r\leq n-2$ one has:
                \begin{equation*}
                \begin{split}
                    \LL_{\mathbf S^m_{0, 0}}\mathbf S^r_{-1, 1} & \simeq \mathbf S^r_{-1, 1} \\
                    \RR_{\mathbf S^r_{-1, 1}}\mathbf S^m_{0, 0} & \simeq \mathbf S^m_{0, 0}.
                \end{split}
                \end{equation*}
            \end{lemma}
            \begin{proof}
                The argument follows the same exact steps of the proof of Lemma \ref{lem:chess_game_exts}. In particular, with the same local analysis, we see that the relevant Ext computations boil down to terms of the form:
                \begin{equation*}
                    \Ext^\bullet(\Sym^m\Uc_{G(2, 2n+\epsilon)}^\vee, \Sym^r\Uc_{G(2, 2n+\epsilon)}^\vee),
                \end{equation*}
                which have no cohomology by Lemma \ref{lem:the_ext_for_the_second_part_of_rule_two}.
            \end{proof}
            \begin{corollary}\label{cor:rule2}
                For $r+1\leq t\leq n-2$ there is a sequence of mutations realizing the following equality:
                \begin{equation}\label{eq:rule_two}
                    \begin{split}
                        \langle \mathbf S^0_{0, 0}, \dots, \mathbf S^0_{t, 0}, \mathbf A^{r-1}_{-1, 1}\rangle = \langle \mathbf A^{r-1}_{-1, 1}, \mathbf A^{r-1}_{0, 0}, \mathbf S^0_{r, 0},\dots, \mathbf S^0_{t, 0},\rangle 
                    \end{split}
                \end{equation}
            \end{corollary}
            \begin{proof}
                By Lemma \ref{lem:chess_game_exts} we can move each block of the shape $\mathbf S^m_{-1, 1}$ immediately to the right of $\mathbf S^0_{m+1, 0}$. Then, we apply Lemma \ref{lem:chess_game_exts} again, mutating the former a further step to the left, where it becomes $\mathbf S^{m+1}_{0,0}$ The last step consists in reordering the blocks so that we obtain the RHS of Equation \ref{eq:rule_two}: this can be done thanks to Lemma \ref{lem:chess_game_exts_II}.
            \end{proof}
            \subsubsection{Arrangement of boxes and semiorthogonal decompositions}
                Hereafter we describe an exceptional collection made of blocks as above by means of an arrangement of boxes, each of them containing a number. The number in a box represents the maximum symmetric power appearing in the associated block, while the position of the box in the table corresponds to the twist, To identify the overall twist, the box corresponding to the twist by $(0,0)$ is grayed out. There are no morphisms from the right to the left in the same row, and from the bottom to the top regardless of the row. For example, one has:
            \begin{equation*}
                \langle \mathbf A ^a_{0,0}, \mathbf A ^b_{1,0}, \mathbf A ^c_{2,0}, \mathbf A ^d_{3,0}, \mathbf A ^e_{1,1}, \mathbf A ^f_{2,1}, \mathbf A ^g_{3,1}, \mathbf A ^h_{4,1}\rangle \hspace{5pt} = \hspace{5pt} 
                \ytableausetup{ boxsize=2em}
            \begin{ytableau} *(lightgray) \scriptstyle a & \scriptstyle b & \scriptstyle c & \scriptstyle d \cr 
                            \scriptstyle e & \scriptstyle f & \scriptstyle g & \scriptstyle h 
            \end{ytableau}
            \end{equation*}
            
            The advantage of this notation is that it allows to visualize useful mutations, even if the semiorthogonal decomposition is exceptionally cumbersome to write.
            
            \begin{lemma}\label{lem:rule1}
                Consider the following diagram:
                
                \begin{equation*}
                    \ytableausetup{ boxsize=2em}
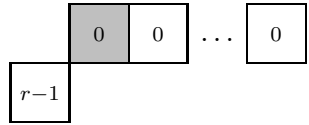
                \begin{ytableau}
                    \none & *(lightgray) \scriptstyle 0& \scriptstyle 0 & \none[\dots] & \scriptstyle 0 \cr 
                    \scriptstyle r-1 
                \end{ytableau}
                \end{equation*}

                where the length $l$ of the first row of zeros satisfies $r\leq l\leq 2n-3+\epsilon$, and $r\leq n-2$ holds. Then, one has:
                
                \begin{equation*}
                    \ytableausetup{ boxsize=2em}
                    \begin{ytableau}
                        \none & *(lightgray) \scriptstyle 0& \scriptstyle 0 & \none[\dots] & \scriptstyle 0 \cr 
                        \scriptstyle r-1 
                    \end{ytableau}
                \hspace{5pt}=
                    \ytableausetup{ boxsize=2em}
                    \begin{ytableau}
                        \none & *(lightgray) \scriptstyle r& \scriptstyle 0 & \none[\dots] & \scriptstyle 0 
                    \end{ytableau}
                \end{equation*}
                
            \end{lemma}
            
            \begin{proof}
                The first diagram describes the following semiorthogonal decomposition:
                
                \begin{equation}
                    \begin{split}
                        \langle  \mathbf A_{0,0} ^0,  \mathbf A ^0_{1,0}, \dots,  \mathbf A ^0_{l,0},  \mathbf A ^{r-1}_{-1,1}\rangle = \langle & \mathbf S^0_{0,0}, \mathbf S^0_{1,0}, \dots\dots\dots\dots\dots, \mathbf S^0_{l,0}, \\
                        & \mathbf S^0_{-1, 1}, \mathbf S^1_{-1,1}, \dots\dots\dots\dots, \mathbf S^{r-1}_{-1,1} \rangle.
                    \end{split}
                \end{equation}
                In light of Lemma \ref{lem:chess_game_exts}, $\mathbf S^0_{-1, 1}$ can be moved to the immediate right of $\mathbf S^0_{1, 0}$. Similarly, for $1\leq m\leq r-1$, we can move the subcategory $\mathbf S^m_{-1, 1}$ right after $\mathbf S^0_{m+1, 0}$. What we get is:
                
                \begin{equation}\label{eq:middlestepmutationslemma}
                    \begin{split}
                        \langle  \mathbf A_{0,0} ^0,  \mathbf A ^0_{1,0}, \dots,  \mathbf A ^0_{l,0},  \mathbf A ^{r-1}_{-1,1}\rangle =\langle & \mathbf S^0_{0,0}, \\
                        &\mathbf S^0_{1,0}, \mathbf S^0_{-1, 1}, 
                        \mathbf S^0_{2,0}, \mathbf S^1_{-1,1},
                        \dots\dots, 
                        \mathbf S^0_{r, 0}, \mathbf S^{r-1}_{-1,1},\\
                        &\mathbf S^0_{r+1, 0}, \dots\dots\dots\dots\dots\dots\dots\dots\dots,  \mathbf S^0_{l,0} \rangle.
                    \end{split}
                \end{equation}
                
                Let us now mutate $\mathbf S^m_{-1,1}$ through the subcategory $\mathbf S^0_{m+1,0}$ for every $m$. We obtain $\LL_{\mathbf S^0_{m+1,1}}\mathbf S^m_{-1,1} = \mathbf S^{m+1}_{0,0}$, and our decomposition becomes:
                \begin{equation*}
                    \langle  \mathbf A_{0,0} ^0,  \mathbf A ^0_{1,0}, \dots,  \mathbf A ^0_{l,0},  \mathbf A ^{k-1}_{-1,1}\rangle = 
                    \langle \mathbf S^0_{0,0}, \mathbf S^1_{0,0}, \mathbf S^0_{1,0}, \mathbf S^2_{0,0},\dots, \mathbf S^k_{0,0}, \mathbf S^0_{k+1,0}, \dots, \mathbf S^0_{l,0}\rangle.
                \end{equation*}
                Observe that by Lemma \ref{lem:reordering_the_block} and the same local computation we did for proving Lemma \ref{lem:chess_game_exts}, we can move all the blocks $\mathbf S^0_{x,y}$ (except for the first one) to the right of all blocks $\mathbf S^{>0}_{x',y'}$. This gives the decomposition 
                \begin{equation*}
                    \langle  \mathbf A_{0,0} ^0,  \mathbf A ^0_{1,0}, \dots,  \mathbf A ^0_{l,0},  \mathbf A ^{k-1}_{-1,1}\rangle = 
                    \langle \mathbf S^0_{0,0}, \mathbf S^1_{0,0}, \dots, \mathbf S^k_{0,0}, \mathbf S^0_{1,0},\dots, \mathbf S^0_{l,0}\rangle,
                \end{equation*}
                which is the one depicted in the second diagram, concluding the proof.
            \end{proof}
    \subsection{The chess game - mutations for the even case}\label{sec:chess_game_fibrations_even}
    In this section we describe the mutations we need to perform to prove Theorem \ref{thm:main_theorem_fibrations_intro} for the ``even'' case, i.e. $\epsilon = 0$. In the notation of Section \ref{sec:arrangement_of_boxes}, we can rewrite the decomposition \ref{eq:sod_for_M_with_line_bundles} as:
    \begin{equation*}
        \begin{split}
            \dbcoh(\Mc) = \langle  i_{1*}\bar q_1\dbcoh(X_1), & \mathbf A^0_{-n,1}, \mathbf A^0_{-n+1,1}, \dots\dots\dots, \mathbf A^0_{n-1,1} \\
                                    & \mathbf A^0_{-n+1,2}, \mathbf A^0_{-n+3,2}, \dots\dots, \mathbf A^0_{n,2} \\
                                    & \hspace{20pt} \vdots \hspace{100pt} \vdots \\
                                    & \mathbf A^0_{n-2,2n-1}, \mathbf A^0_{n-1,2n-1}, \dots, \mathbf A^0_{n,2n+3} \rangle\\
        \end{split}
    \end{equation*}
    
    As we discussed above in Section \ref{sec:the_chess_game_rules}, $\dbcoh(X_1)^\perp$ can be rewritten as an arrangement of blocks, or a ``chessboard''. In the remainder of this section, we choose $n = 6$ to depict the chessboard moves we perform, while the argument itself will be presented in full generality.
    
    \begin{adjustbox}{width=.8\textwidth,center}
        \ytableausetup{boxsize=1.5em}
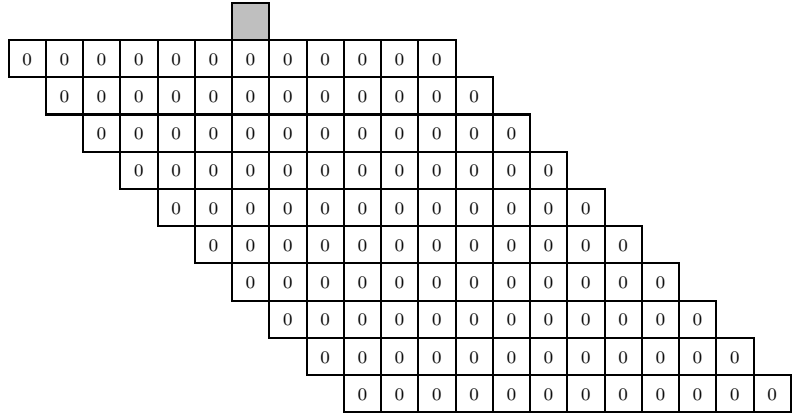
        \begin{ytableau} 
        \none & \none & \none & \none & \none & \none &  *(lightgray) \cr 
        \scriptstyle  0 &\scriptstyle 0  &\scriptstyle 0  &\scriptstyle 0  &\scriptstyle 0  &\scriptstyle 0  &\scriptstyle 0  &\scriptstyle 0  &\scriptstyle 0  &\scriptstyle 0  &\scriptstyle 0  &\scriptstyle 0  &\none &\none &\none &\none &\none &\none &\none &\none &\none &\none &\none & \none &\none \cr 
        \none &\scriptstyle 0  &\scriptstyle 0  &\scriptstyle 0  &\scriptstyle 0  &\scriptstyle 0  &\scriptstyle 0  &\scriptstyle 0  &\scriptstyle 0  &\scriptstyle 0  &\scriptstyle 0  &\scriptstyle 0  &\scriptstyle 0 &\none &\none &\none &\none &\none &\none &\none &\none &\none &\none &\none & \none \cr 
        \none &\none & \scriptstyle 0  &\scriptstyle 0  &\scriptstyle 0  &\scriptstyle 0  &\scriptstyle 0  &\scriptstyle 0  &\scriptstyle 0  &\scriptstyle 0  &\scriptstyle 0  &\scriptstyle 0  &\scriptstyle 0  &\scriptstyle 0 &\none &\none &\none &\none &\none &\none &\none &\none &\none &\none &\none\cr 
        \none & \none &\none &\scriptstyle 0  &\scriptstyle 0  &\scriptstyle 0  &\scriptstyle 0  &\scriptstyle 0  &\scriptstyle 0  &\scriptstyle 0  &\scriptstyle 0  &\scriptstyle 0  &\scriptstyle 0  &\scriptstyle 0  &\scriptstyle 0 &\none &\none &\none &\none &\none &\none &\none &\none &\none &\none \cr 
        \none & \none & \none &\none &\scriptstyle 0  &\scriptstyle 0  &\scriptstyle 0  &\scriptstyle 0  &\scriptstyle 0  &\scriptstyle 0  &\scriptstyle 0  &\scriptstyle 0  &\scriptstyle 0  &\scriptstyle 0  &\scriptstyle 0  &\scriptstyle 0 &\none &\none &\none &\none &\none &\none &\none &\none &\none \cr
        \none &\none & \none & \none &\none &\scriptstyle 0  &\scriptstyle 0  &\scriptstyle 0  &\scriptstyle 0  &\scriptstyle 0  &\scriptstyle 0  &\scriptstyle 0  &\scriptstyle 0  &\scriptstyle 0  &\scriptstyle 0  &\scriptstyle 0  &\scriptstyle 0 &\none &\none &\none &\none &\none &\none &\none &\none \cr 
        \none &\none &\none& \none & \none &\none &\scriptstyle 0  &\scriptstyle 0  &\scriptstyle 0  &\scriptstyle 0  &\scriptstyle 0  &\scriptstyle 0  &\scriptstyle 0  &\scriptstyle 0  &\scriptstyle 0  &\scriptstyle 0  &\scriptstyle 0  &\scriptstyle 0 &\none &\none &\none &\none &\none &\none &\none \cr 
        \none &\none &\none &\none & \none & \none &\none &\scriptstyle 0  &\scriptstyle 0  &\scriptstyle 0  &\scriptstyle 0  &\scriptstyle 0  &\scriptstyle 0  &\scriptstyle 0  &\scriptstyle 0  &\scriptstyle 0  &\scriptstyle 0  &\scriptstyle 0  &\scriptstyle 0 &\none &\none &\none &\none &\none &\none \cr 
        \none &\none &\none &\none &\none & \none & \none &\none &\scriptstyle 0  &\scriptstyle 0  &\scriptstyle 0  &\scriptstyle 0  &\scriptstyle 0  &\scriptstyle 0  &\scriptstyle 0  &\scriptstyle 0  &\scriptstyle 0  &\scriptstyle 0  &\scriptstyle 0  &\scriptstyle 0 &\none &\none &\none &\none &\none \cr 
        \none &\none &\none &\none &\none &\none & \none & \none &\none  &\scriptstyle 0  &\scriptstyle 0  &\scriptstyle 0  &\scriptstyle 0  &\scriptstyle 0  &\scriptstyle 0  &\scriptstyle 0  &\scriptstyle 0  &\scriptstyle 0  &\scriptstyle 0  &\scriptstyle 0  &\scriptstyle 0 &\none &\none &\none &\none \cr 
        \end{ytableau}
    \end{adjustbox}

    Observe that each row represents a different twist of the pullback of $\dbcoh(\Gc r(1, \Ec))$: hence, we can use the Serre functor of such category to ``translate'' the row horizontally. We eventually get the following diagram, where some blocks are highlighted for further convenience (the height of the yellow area is $n-2$ and the length is $n-1$):

    \begin{adjustbox}{width=.5\textwidth,center}
        \ytableausetup{boxsize=2em}
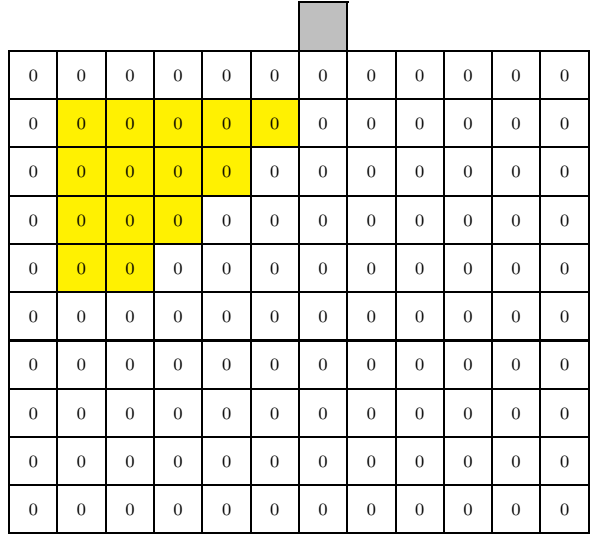
        \begin{ytableau}
        \none & \none & \none & \none & \none & \none & *(lightgray) \cr 
        \scriptstyle  0 &\scriptstyle 0  &\scriptstyle 0  &\scriptstyle 0  &\scriptstyle 0    &\scriptstyle 0  &\scriptstyle 0  &\scriptstyle 0  &\scriptstyle 0  &\scriptstyle 0  &\scriptstyle 0  &\scriptstyle 0  \cr 
        \scriptstyle  0 &*(yellow)\scriptstyle 0  &*(yellow)\scriptstyle 0  &*(yellow)\scriptstyle 0  &*(yellow)\scriptstyle 0  &*(yellow)\scriptstyle 0  &\scriptstyle 0  &\scriptstyle 0 &\scriptstyle 0  &\scriptstyle 0  &\scriptstyle 0  &\scriptstyle 0    \cr 
        \scriptstyle  0 &*(yellow)\scriptstyle 0  &*(yellow)\scriptstyle 0  &*(yellow)\scriptstyle 0  &*(yellow)\scriptstyle 0  &\scriptstyle 0  &\scriptstyle 0  &\scriptstyle 0  &\scriptstyle 0  &\scriptstyle 0  &\scriptstyle 0    &\scriptstyle 0  \cr 
        \scriptstyle  0 &*(yellow)\scriptstyle 0  &*(yellow)\scriptstyle 0  &*(yellow)\scriptstyle 0  &\scriptstyle 0  &\scriptstyle 0  &\scriptstyle 0  &\scriptstyle 0  &\scriptstyle 0  &\scriptstyle 0    &\scriptstyle 0  &\scriptstyle 0  \cr 
        \scriptstyle  0 &*(yellow)\scriptstyle 0  &*(yellow)\scriptstyle 0    &\scriptstyle 0  &\scriptstyle 0  &\scriptstyle 0  &\scriptstyle 0    &\scriptstyle 0  &\scriptstyle 0  &\scriptstyle 0  &\scriptstyle 0  &\scriptstyle 0\cr 
        \scriptstyle  0 &\scriptstyle 0  &\scriptstyle 0  &\scriptstyle 0  &\scriptstyle 0  &\scriptstyle 0  &\scriptstyle 0  &\scriptstyle 0    &\scriptstyle 0  &\scriptstyle 0  &\scriptstyle 0  &\scriptstyle 0  \cr 
        \scriptstyle  0 &\scriptstyle 0  &\scriptstyle 0  &\scriptstyle 0  &\scriptstyle 0  &\scriptstyle 0  &\scriptstyle 0    &\scriptstyle 0  &\scriptstyle 0  &\scriptstyle 0  &\scriptstyle 0  &\scriptstyle 0  \cr 
        \scriptstyle  0 &\scriptstyle 0  &\scriptstyle 0  &\scriptstyle 0  &\scriptstyle 0    &\scriptstyle 0  &\scriptstyle 0  &\scriptstyle 0  &\scriptstyle 0  &\scriptstyle 0  &\scriptstyle 0  &\scriptstyle 0  \cr 
        \scriptstyle  0 &\scriptstyle 0  &\scriptstyle 0  &\scriptstyle 0  &\scriptstyle 0    &\scriptstyle 0  &\scriptstyle 0  &\scriptstyle 0  &\scriptstyle 0  &\scriptstyle 0  &\scriptstyle 0  &\scriptstyle 0  \cr 
        \scriptstyle  0 &\scriptstyle 0  &\scriptstyle 0  &\scriptstyle 0  &\scriptstyle 0    &\scriptstyle 0  &\scriptstyle 0  &\scriptstyle 0  &\scriptstyle 0  &\scriptstyle 0  &\scriptstyle 0  &\scriptstyle 0  \cr 
        \end{ytableau}
    \end{adjustbox}
    
    \subsubsection{First upward phase} Let us apply the rule described in Lemma \ref{lem:rule1} to the shortest yellow row, after moving it two steps to the right by mutating the block it passes through. Since we are not interested in the explicit description of such subcategory, we will denote it (and the similar pieces we will produce in the following) by $\times$. We find:

    \begin{adjustbox}{width=.5\textwidth,center}
        \ytableausetup{boxsize=2em}
        \begin{ytableau}
        \none & \none & \none & \none & \none & \none & *(lightgray) \cr 
        \scriptstyle  0 &\scriptstyle 0  &\scriptstyle 0  &\scriptstyle 0  &\scriptstyle 0    &\scriptstyle 0  &\scriptstyle 0  &\scriptstyle 0  &\scriptstyle 0  &\scriptstyle 0  &\scriptstyle 0  &\scriptstyle 0  \cr 
        \scriptstyle  0 &*(yellow)\scriptstyle 0  &*(yellow)\scriptstyle 0  &*(yellow)\scriptstyle 0  &*(yellow)\scriptstyle 0  &*(yellow)\scriptstyle 0  &\scriptstyle 0  &\scriptstyle 0 &\scriptstyle 0  &\scriptstyle 0  &\scriptstyle 0  &\scriptstyle 0    \cr 
        \scriptstyle  0 &*(yellow)\scriptstyle 0  &*(yellow)\scriptstyle 0  &*(yellow)\scriptstyle 0  &*(yellow)\scriptstyle 0  &\scriptstyle 0  &\scriptstyle 0  &\scriptstyle 0  &\scriptstyle 0  &\scriptstyle 0  &\scriptstyle 0    &\scriptstyle 0  \cr 
        \scriptstyle  0 &*(yellow)\scriptstyle 0  &*(yellow)\scriptstyle 1  &*(yellow)\scriptstyle 1  &\scriptstyle 0  &\scriptstyle 0  &\scriptstyle 0  &\scriptstyle 0  &\scriptstyle 0  &\scriptstyle 0    &\scriptstyle 0  &\scriptstyle 0  \cr 
        \scriptstyle  \times & \none  & \none  &\scriptstyle 0  &\scriptstyle 0  &\scriptstyle 0  &\scriptstyle 0    &\scriptstyle 0  &\scriptstyle 0  &\scriptstyle 0  &\scriptstyle 0  &\scriptstyle 0\cr 
        \scriptstyle  0 &\scriptstyle 0  &\scriptstyle 0  &\scriptstyle 0  &\scriptstyle 0  &\scriptstyle 0  &\scriptstyle 0  &\scriptstyle 0    &\scriptstyle 0  &\scriptstyle 0  &\scriptstyle 0  &\scriptstyle 0  \cr 
        \scriptstyle  0 &\scriptstyle 0  &\scriptstyle 0  &\scriptstyle 0  &\scriptstyle 0  &\scriptstyle 0  &\scriptstyle 0    &\scriptstyle 0  &\scriptstyle 0  &\scriptstyle 0  &\scriptstyle 0  &\scriptstyle 0  \cr 
        \scriptstyle  0 &\scriptstyle 0  &\scriptstyle 0  &\scriptstyle 0  &\scriptstyle 0    &\scriptstyle 0  &\scriptstyle 0  &\scriptstyle 0  &\scriptstyle 0  &\scriptstyle 0  &\scriptstyle 0  &\scriptstyle 0  \cr 
        \scriptstyle  0 &\scriptstyle 0  &\scriptstyle 0  &\scriptstyle 0  &\scriptstyle 0    &\scriptstyle 0  &\scriptstyle 0  &\scriptstyle 0  &\scriptstyle 0  &\scriptstyle 0  &\scriptstyle 0  &\scriptstyle 0  \cr 
        \scriptstyle  0 &\scriptstyle 0  &\scriptstyle 0  &\scriptstyle 0  &\scriptstyle 0    &\scriptstyle 0  &\scriptstyle 0  &\scriptstyle 0  &\scriptstyle 0  &\scriptstyle 0  &\scriptstyle 0  &\scriptstyle 0  \cr 
        \end{ytableau}
    \end{adjustbox}

    By applying the same step to each yellow row, progressivley from the shortest to the longest, we find the following ``chessboard'':

    \begin{adjustbox}{width=.5\textwidth,center}
        \ytableausetup{boxsize=2em}
        \begin{ytableau}
        \none & \none & \none & \none & \none & \none & *(lightgray) \cr 
        \scriptstyle  0 &\scriptstyle 0  &\scriptstyle 1  &\scriptstyle 2  &\scriptstyle 3    &\scriptstyle 4  &\scriptstyle 4  &\scriptstyle 0  &\scriptstyle 0  &\scriptstyle 0  &\scriptstyle 0  &\scriptstyle 0  \cr 
        \scriptstyle  \times & \none  & \none  & \none & \none & \none &\scriptstyle 0  &\scriptstyle 0 &\scriptstyle 0  &\scriptstyle 0  &\scriptstyle 0  &\scriptstyle 0    \cr 
        \scriptstyle  \times & \none & \none & \none & \none &\scriptstyle 0  &\scriptstyle 0  &\scriptstyle 0  &\scriptstyle 0  &\scriptstyle 0  &\scriptstyle 0    &\scriptstyle 0  \cr 
        \scriptstyle  \times & \none & \none & \none &\scriptstyle 0  &\scriptstyle 0  &\scriptstyle 0  &\scriptstyle 0  &\scriptstyle 0  &\scriptstyle 0    &\scriptstyle 0  &\scriptstyle 0  \cr 
        \scriptstyle  \times & \none & \none &\scriptstyle 0  &\scriptstyle 0  &\scriptstyle 0  &\scriptstyle 0    &\scriptstyle 0  &\scriptstyle 0  &\scriptstyle 0  &\scriptstyle 0  &\scriptstyle 0\cr 
        \scriptstyle  0 &\scriptstyle 0  &\scriptstyle 0  &\scriptstyle 0  &\scriptstyle 0  &\scriptstyle 0  &\scriptstyle 0  &\scriptstyle 0    &\scriptstyle 0  &\scriptstyle 0  &\scriptstyle 0  &\scriptstyle 0  \cr 
        \scriptstyle  0 &\scriptstyle 0  &\scriptstyle 0  &\scriptstyle 0  &\scriptstyle 0  &\scriptstyle 0  &\scriptstyle 0    &\scriptstyle 0  &\scriptstyle 0  &\scriptstyle 0  &\scriptstyle 0  &\scriptstyle 0  \cr 
        \scriptstyle  0 &\scriptstyle 0  &\scriptstyle 0  &\scriptstyle 0  &\scriptstyle 0    &\scriptstyle 0  &\scriptstyle 0  &\scriptstyle 0  &\scriptstyle 0  &\scriptstyle 0  &\scriptstyle 0  &\scriptstyle 0  \cr 
        \scriptstyle  0 &\scriptstyle 0  &\scriptstyle 0  &\scriptstyle 0  &\scriptstyle 0    &\scriptstyle 0  &\scriptstyle 0  &\scriptstyle 0  &\scriptstyle 0  &\scriptstyle 0  &\scriptstyle 0  &\scriptstyle 0  \cr 
        \scriptstyle  0 &\scriptstyle 0  &\scriptstyle 0  &\scriptstyle 0  &\scriptstyle 0    &\scriptstyle 0  &\scriptstyle 0  &\scriptstyle 0  &\scriptstyle 0  &\scriptstyle 0  &\scriptstyle 0  &\scriptstyle 0  \cr 
        \end{ytableau}
    \end{adjustbox}
    
    where the last nonzero number in the first row, in the general case, is $n-2$. Let us now apply the Serre functor to the segment of the first row terminating with the two ``$n-2$`` blocks. We will introduce new colored areas to simplify the exposition of the next phase.

     \begin{adjustbox}{width=.5\textwidth,center}
        \ytableausetup{boxsize=2em}
        \begin{ytableau}
        \none & \none & \none & \none & \none & \none & *(lightgray) \cr 
        \none & \none & \none & \none & \none & \none & \scriptstyle 4 &\scriptstyle 0  &\scriptstyle 0  &\scriptstyle 0  &\scriptstyle 0  &\scriptstyle 0  \cr 
        \scriptstyle  \times & \none  & \none  & \none & \none & \none &\scriptstyle 0  &\scriptstyle 0 &\scriptstyle 0  &\scriptstyle 0  &\scriptstyle 0  &\scriptstyle 0    \cr 
        \scriptstyle  \times & \none & \none & \none & \none &*(yellow)\scriptstyle 0  &\scriptstyle 0  &\scriptstyle 0  &\scriptstyle 0  &\scriptstyle 0  &\scriptstyle 0    &\scriptstyle 0  \cr 
        \scriptstyle  \times & \none & \none & \none &*(yellow)\scriptstyle 0  &*(yellow)\scriptstyle 0  &\scriptstyle 0  &\scriptstyle 0  &\scriptstyle 0  &\scriptstyle 0    &\scriptstyle 0  &\scriptstyle 0  \cr 
        \scriptstyle  \times & \none & \none &*(yellow)\scriptstyle 0  &*(yellow)\scriptstyle 0  &*(yellow)\scriptstyle 0  &\scriptstyle 0    &\scriptstyle 0  &\scriptstyle 0  &\scriptstyle 0  &\scriptstyle 0  &\scriptstyle 0\cr 
        \scriptstyle  0 &*(orange)\scriptstyle 0  &*(yellow)\scriptstyle 0  &*(yellow)\scriptstyle 0  &*(yellow)\scriptstyle 0  &*(yellow)\scriptstyle 0  &\scriptstyle 0  &\scriptstyle 0    &\scriptstyle 0  &\scriptstyle 0  &\scriptstyle 0  &\scriptstyle 0  \cr 
        \scriptstyle  0 &*(yellow)\scriptstyle 0  &*(yellow)\scriptstyle 0  &*(yellow)\scriptstyle 0  &*(yellow)\scriptstyle 0  &*(yellow)\scriptstyle 0  &\scriptstyle 0    &\scriptstyle 0  &\scriptstyle 0  &\scriptstyle 0  &\scriptstyle 0  &\scriptstyle 0  \cr 
        \scriptstyle  0 &*(yellow)\scriptstyle 0  &*(yellow)\scriptstyle 0  &*(yellow)\scriptstyle 0  &*(yellow)\scriptstyle 0    &*(yellow)\scriptstyle 0  &\scriptstyle 0  &\scriptstyle 0  &\scriptstyle 0  &\scriptstyle 0  &\scriptstyle 0  &\scriptstyle 0  \cr 
        \scriptstyle  0 &*(yellow)\scriptstyle 0  &*(yellow)\scriptstyle 0  &*(yellow)\scriptstyle 0  &*(yellow)\scriptstyle 0    &*(yellow)\scriptstyle 0  &\scriptstyle 0  &\scriptstyle 0  &\scriptstyle 0  &\scriptstyle 0  &\scriptstyle 0  &\scriptstyle 0  \cr 
        \scriptstyle  0 &*(yellow)\scriptstyle 0  &*(yellow)\scriptstyle 0  &*(yellow)\scriptstyle 0  &*(yellow)\scriptstyle 0    &*(yellow)\scriptstyle 0  &\scriptstyle 0  &\scriptstyle 0  &\scriptstyle 0  &\scriptstyle 0  &\scriptstyle 0  &\scriptstyle 0  \cr 
        \none & *(yellow)\scriptstyle  0 &*(yellow)\scriptstyle 0  &*(yellow)\scriptstyle 1  &*(yellow)\scriptstyle 2  &*(yellow)\scriptstyle 3    &\scriptstyle 4  
        \end{ytableau}
    \end{adjustbox}

    \subsubsection{Second upward phase} Let us mutate, as in the first upward phase, the first yellow row using the rule described in Lemma \ref{lem:rule1}.

    \begin{adjustbox}{width=.5\textwidth,center}
        \ytableausetup{boxsize=2em}
        \begin{ytableau}
        \none & \none & \none & \none & \none & \none & *(lightgray) \cr 
        \none & \none & \none & \none & \none & \none & \scriptstyle 4 &\scriptstyle 0  &\scriptstyle 0  &\scriptstyle 0  &\scriptstyle 0  &\scriptstyle 0  \cr 
        \scriptstyle  \times & \none  & \none  & \none & \none & \none &\scriptstyle 0  &\scriptstyle 0 &\scriptstyle 0  &\scriptstyle 0  &\scriptstyle 0  &\scriptstyle 0    \cr 
        \scriptstyle  \times & \none & \none & \none & \none &*(yellow)\scriptstyle 0  &\scriptstyle 0  &\scriptstyle 0  &\scriptstyle 0  &\scriptstyle 0  &\scriptstyle 0    &\scriptstyle 0  \cr 
        \scriptstyle  \times & \none & \none & \none &*(yellow)\scriptstyle 0  &*(yellow)\scriptstyle 0  &\scriptstyle 0  &\scriptstyle 0  &\scriptstyle 0  &\scriptstyle 0    &\scriptstyle 0  &\scriptstyle 0  \cr 
        \scriptstyle  \times & \none & \none &*(yellow)\scriptstyle 0  &*(yellow)\scriptstyle 0  &*(yellow)\scriptstyle 0  &\scriptstyle 0    &\scriptstyle 0  &\scriptstyle 0  &\scriptstyle 0  &\scriptstyle 0  &\scriptstyle 0\cr 
        \scriptstyle  0 &*(orange)\scriptstyle 0  &*(yellow)\scriptstyle 0  &*(yellow)\scriptstyle 0  &*(yellow)\scriptstyle 0  &*(yellow)\scriptstyle 0  &\scriptstyle 0  &\scriptstyle 0    &\scriptstyle 0  &\scriptstyle 0  &\scriptstyle 0  &\scriptstyle 0  \cr 
        \scriptstyle  0 &*(yellow)\scriptstyle 0  &*(yellow)\scriptstyle 0  &*(yellow)\scriptstyle 0  &*(yellow)\scriptstyle 0  &*(yellow)\scriptstyle 0  &\scriptstyle 0    &\scriptstyle 0  &\scriptstyle 0  &\scriptstyle 0  &\scriptstyle 0  &\scriptstyle 0  \cr 
        \scriptstyle  0 &*(yellow)\scriptstyle 0  &*(yellow)\scriptstyle 0  &*(yellow)\scriptstyle 0  &*(yellow)\scriptstyle 0    &*(yellow)\scriptstyle 0  &\scriptstyle 0  &\scriptstyle 0  &\scriptstyle 0  &\scriptstyle 0  &\scriptstyle 0  &\scriptstyle 0  \cr 
        \scriptstyle  0 &*(yellow)\scriptstyle 0  &*(yellow)\scriptstyle 0  &*(yellow)\scriptstyle 0  &*(yellow)\scriptstyle 0    &*(yellow)\scriptstyle 0  &\scriptstyle 0  &\scriptstyle 0  &\scriptstyle 0  &\scriptstyle 0  &\scriptstyle 0  &\scriptstyle 0  \cr 
        \scriptstyle  0 &*(yellow)\scriptstyle 0  &*(yellow)\scriptstyle 1  &*(yellow)\scriptstyle 1  &*(yellow)\scriptstyle 2    &*(yellow)\scriptstyle 3  &\scriptstyle 4  &\scriptstyle 0  &\scriptstyle 0  &\scriptstyle 0  &\scriptstyle 0  &\scriptstyle 0  \cr 
        \none & \none  0 &\none  &\none  &\none  &\none    &\scriptstyle 4  
        \end{ytableau}
    \end{adjustbox}

    After iterating this operation until we erase all the yellow blocks, we obtain the following final board:

    \begin{adjustbox}{width=.5\textwidth,center}
        \ytableausetup{boxsize=2em}
        \begin{ytableau}
        \none & \none & \none & \none & \none & \none & *(lightgray) \cr 
        \none & \none & \none & \none & \none & \none & \scriptstyle 5 &\scriptstyle 0  &\scriptstyle 0  &\scriptstyle 0  &\scriptstyle 0  &\scriptstyle 0  \cr 
        \scriptstyle  \times & \none  & \none  & \none & \none & \none &\scriptstyle 5  &\scriptstyle 0 &\scriptstyle 0  &\scriptstyle 0  &\scriptstyle 0  &\scriptstyle 0    \cr 
        \scriptstyle  \times & \none & \none & \none & \none & \none   &\scriptstyle 5  &\scriptstyle 0  &\scriptstyle 0  &\scriptstyle 0  &\scriptstyle 0    &\scriptstyle 0  \cr 
        \scriptstyle  \times & \none & \none & \none & \none   & \none   &\scriptstyle 5  &\scriptstyle 0  &\scriptstyle 0  &\scriptstyle 0    &\scriptstyle 0  &\scriptstyle 0  \cr 
        \scriptstyle  \times & \none   & \none   & \none   & \none   & \none   &\scriptstyle 5  &\scriptstyle 0    &\scriptstyle 0  &\scriptstyle 0  &\scriptstyle 0  &\scriptstyle 0  \cr 
        \scriptstyle  \times & \none   & \none   & \none   & \none   & \none   &\scriptstyle 5    &\scriptstyle 0  &\scriptstyle 0  &\scriptstyle 0  &\scriptstyle 0  &\scriptstyle 0  \cr 
        \scriptstyle  \times & \none   & \none   & \none   & \none     & \none   &\scriptstyle 4  &\scriptstyle 0  &\scriptstyle 0  &\scriptstyle 0  &\scriptstyle 0  &\scriptstyle 0  \cr 
        \scriptstyle  \times & \none   & \none   & \none   & \none     & \none   &\scriptstyle 4  &\scriptstyle 0  &\scriptstyle 0  &\scriptstyle 0  &\scriptstyle 0  &\scriptstyle 0  \cr 
        \scriptstyle  \times & \none   & \none   & \none   & \none     & \none   &\scriptstyle 4  &\scriptstyle 0  &\scriptstyle 0  &\scriptstyle 0  &\scriptstyle 0  &\scriptstyle 0  \cr 
        \scriptstyle  \times & \none   & \none  & \none  & \none    & \none  &\scriptstyle 4  &\scriptstyle 0  &\scriptstyle 0  &\scriptstyle 0  &\scriptstyle 0  &\scriptstyle 0  \cr 
        \none & \none  0 &\none  &\none  &\none  &\none    &\scriptstyle 4  
        \end{ytableau}
    \end{adjustbox}

    Note that the content of the upmost box is $n-1$, and not $n-2$. This is because the subcategory denoted by the orange box, after all this process, gets mutated to $\mathbf S^{n-1}_{0, 1}$. Now let us apply Corollary \ref{cor:rule2} to the last ``$n-2$'' block and the half-row of zeros immediately above. Finally, we apply the Serre funcotr to the resulting new ``$n-2$'' block, obtaining the final chessboard:
    
    \begin{adjustbox}{width=.5\textwidth,center}
        \ytableausetup{boxsize=2em}
        \begin{ytableau}
        \none & \none & \none & \none & \none & \none & *(lightgray) \scriptstyle 4 \cr 
        \none & \none & \none & \none & \none & \none & \scriptstyle 5 &\scriptstyle 0  &\scriptstyle 0  &\scriptstyle 0  &\scriptstyle 0  &\scriptstyle 0  \cr 
        \scriptstyle  \times & \none  & \none  & \none & \none & \none &\scriptstyle 5  &\scriptstyle 0 &\scriptstyle 0  &\scriptstyle 0  &\scriptstyle 0  &\scriptstyle 0    \cr 
        \scriptstyle  \times & \none & \none & \none & \none & \none   &\scriptstyle 5  &\scriptstyle 0  &\scriptstyle 0  &\scriptstyle 0  &\scriptstyle 0    &\scriptstyle 0  \cr 
        \scriptstyle  \times & \none & \none & \none & \none   & \none   &\scriptstyle 5  &\scriptstyle 0  &\scriptstyle 0  &\scriptstyle 0    &\scriptstyle 0  &\scriptstyle 0  \cr 
        \scriptstyle  \times & \none   & \none   & \none   & \none   & \none   &\scriptstyle 5  &\scriptstyle 0    &\scriptstyle 0  &\scriptstyle 0  &\scriptstyle 0  &\scriptstyle 0  \cr 
        \scriptstyle  \times & \none   & \none   & \none   & \none   & \none   &\scriptstyle 5    &\scriptstyle 0  &\scriptstyle 0  &\scriptstyle 0  &\scriptstyle 0  &\scriptstyle 0  \cr 
        \scriptstyle  \times & \none   & \none   & \none   & \none     & \none   &\scriptstyle 4  &\scriptstyle 0  &\scriptstyle 0  &\scriptstyle 0  &\scriptstyle 0  &\scriptstyle 0  \cr 
        \scriptstyle  \times & \none   & \none   & \none   & \none     & \none   &\scriptstyle 4  &\scriptstyle 0  &\scriptstyle 0  &\scriptstyle 0  &\scriptstyle 0  &\scriptstyle 0  \cr 
        \scriptstyle  \times & \none   & \none   & \none   & \none     & \none   &\scriptstyle 4  &\scriptstyle 0  &\scriptstyle 0  &\scriptstyle 0  &\scriptstyle 0  &\scriptstyle 0  \cr 
        \scriptstyle  \times & \none   & \none  & \none  & \none    & \none  &\scriptstyle 4  &\none 
        &\none   &\none   &\none   &\none   \cr 
        \none & \none  0 &\none  &\none  &\none  &\none    &\scriptstyle 4  
        \end{ytableau}
    \end{adjustbox}
    
    The direct consequence of this sequence of mutations is the following, which settles the construction of the derived embedding for the even case:
    \begin{proposition}\label{prop:main_theorem_fibrations_even}
        Let $X_1$ and $X_2$ be as described in Section \ref{sec:cover_and_fano_fibration}, for $\epsilon = 0$. Then there is a fully faithful functor $\Psi:\dbcoh(X_1)\xhookrightarrow{\,\,\,\,\,}\dbcoh(X_2)$. Moreover, $\dbcoh(X_2)$ admits a semiorthogonal decomposition as follows:
        \begin{equation*}
            \dbcoh(X_2) = \langle \Psi(\dbcoh(X_1)), F_1, \dots, F_{2n^2-n}\rangle
        \end{equation*}
        where, for all $i$, $F_i$ is an admissible subcategory given by the image of $\dbcoh(B)$ through a fully faithful functor.
    \end{proposition}
    \begin{proof}
        In the notation of Section \ref{sec:arrangement_of_boxes}, the last chessboard  essentially says that:
        \begin{equation}\label{eq:collection_final_chesssboard_even}
            \begin{split}
                \Phi i_{1*}\bar q_1\dbcoh(X_1)^\perp = \langle  
                                        & \hspace{12pt} \mathbf A^{n-2}_{0,0}, \\
                                        & \hspace{12pt}\mathbf A^{n-1}_{0,1}, \mathbf A^0_{1,1}, \dots\dots\dots, \mathbf A^0_{n-1,1} \\
                                        & \times, \mathbf A^{n-1}_{0,2}, \mathbf A^0_{1,2}, \dots\dots\dots, \mathbf A^0_{n-1,2} \\
                                        & \hspace{20pt} \vdots \hspace{100pt} \vdots \\
                                        & \times, \mathbf A^{n-1}_{0,n}, \mathbf A^0_{1,n+1}, \dots\dots, \mathbf A^0_{n-1,n+1} \\
                                        & \times, \mathbf A^{n-2}_{0,n+1}, \mathbf A^0_{1,n+2}, \dots\dots, \mathbf A^0_{n-1,n+2} \\
                                        & \hspace{20pt} \vdots \hspace{100pt} \vdots \\
                                        & \times, \mathbf A^{n-2}_{0,2n-2}, \mathbf A^0_{1,2n-2}, \dots\dots, \mathbf A^0_{n-1,2n-2} \\
                                        & \hspace{12pt} \mathbf A^{n-2}_{0,2n-1}\rangle\\
            \end{split}
        \end{equation}
        where $\Phi$ is the functor induced by mutations. On the other hand, by Serre functor, one has:
        \begin{equation*}
            \begin{split}
                q_2^* \dbcoh(\Gc r(2, \Ec))\otimes\Oc(0, 1)   & = \langle  \mathbf A^{n-1}_{0,1}, \dots, \mathbf A^{n-1}_{0,n}, \mathbf A^{n-2}_{0,n+1}, \dots, \mathbf A^{n-2}_{0,2n-2}, \mathbf A^{n-2}_{0,2n-1},\mathbf A^{n-2}_{0,2n}, \rangle \\
                                        & = \langle \mathbf A^{n-2}_{0,0}, \mathbf A^{n-1}_{0,1}, \dots, \mathbf A^{n-1}_{0,n}, \mathbf A^{n-2}_{0,n+1}, \dots, \mathbf A^{n-2}_{0,2n-2}, \mathbf A^{n-2}_{0,2n-1} \rangle,
            \end{split}
        \end{equation*}
        and the blocks in the second line all appear in the collection \ref{eq:collection_final_chesssboard_even}. Hence we can move all of them to the end of the semiorthogonal decomposition (mutating the blocks in between accordingly), and we finally find:
        \begin{equation}\label{eq:collection_final_sod_even}
            \dbcoh(\Mc) = \langle \Phi i_{1*}\bar q_1\dbcoh(X_1), F_1, \dots, F_{2n^2-n}, q_2^*\dbcoh(\Gc r(2, \Ec))\rangle
        \end{equation}
        where $F_i:= p_1^*r_1^*\dbcoh(B)\otimes E_i$, for $E_i$ an exceptional object.\footnote{A more explicit description of the $E_i$'s can be obtained by computing the mutations which lead from the collection \ref{eq:collection_final_chesssboard_even} to \ref{eq:collection_final_sod_even}. However, these mutations would be exceptionally cumbersome to write, burdening an already heavy notation.} The proof is complete once we set $\Psi:= \Phi i_{1*}\bar q_1$.
    \end{proof}
    \subsection{The chess game - mutations for the odd case}\label{sec:chess_game_fibrations_odd}
    We now address the odd case, i.e. $\epsilon = 1$. As above, we choose to write the chessboards for $n=6$ while we describe the argument in general. We start with:
    
    \begin{adjustbox}{width=.8\textwidth,center}
        \ytableausetup{boxsize=1.5em}
        \begin{ytableau} 
        \none & \none & \none & \none & \none & \none &  *(lightgray) \cr 
        \scriptstyle  0 &\scriptstyle 0  &\scriptstyle 0  &\scriptstyle 0  &\scriptstyle 0  &\scriptstyle 0  &\scriptstyle 0  &\scriptstyle 0  &\scriptstyle 0  &\scriptstyle 0  &\scriptstyle 0  &\scriptstyle 0  &\scriptstyle 0 &\none &\none &\none &\none &\none &\none &\none &\none &\none &\none &\none & \none &\none \cr 
        \none &\scriptstyle 0  &\scriptstyle 0  &\scriptstyle 0  &\scriptstyle 0  &\scriptstyle 0  &\scriptstyle 0  &\scriptstyle 0  &\scriptstyle 0  &\scriptstyle 0  &\scriptstyle 0  &\scriptstyle 0  &\scriptstyle 0 &\scriptstyle 0 &\none &\none &\none &\none &\none &\none &\none &\none &\none &\none &\none & \none \cr 
        \none &\none & \scriptstyle 0  &\scriptstyle 0  &\scriptstyle 0  &\scriptstyle 0  &\scriptstyle 0  &\scriptstyle 0  &\scriptstyle 0  &\scriptstyle 0  &\scriptstyle 0  &\scriptstyle 0  &\scriptstyle 0  &\scriptstyle 0 &\scriptstyle 0 &\none &\none &\none &\none &\none &\none &\none &\none &\none &\none &\none\cr 
        \none & \none &\none &\scriptstyle 0  &\scriptstyle 0  &\scriptstyle 0  &\scriptstyle 0  &\scriptstyle 0  &\scriptstyle 0  &\scriptstyle 0  &\scriptstyle 0  &\scriptstyle 0  &\scriptstyle 0  &\scriptstyle 0  &\scriptstyle 0 &\scriptstyle 0 &\none &\none &\none &\none &\none &\none &\none &\none &\none &\none \cr 
        \none & \none & \none &\none &\scriptstyle 0  &\scriptstyle 0  &\scriptstyle 0  &\scriptstyle 0  &\scriptstyle 0  &\scriptstyle 0  &\scriptstyle 0  &\scriptstyle 0  &\scriptstyle 0  &\scriptstyle 0  &\scriptstyle 0  &\scriptstyle 0 &\scriptstyle 0 &\none &\none &\none &\none &\none &\none &\none &\none &\none \cr
        \none &\none & \none & \none &\none &\scriptstyle 0  &\scriptstyle 0  &\scriptstyle 0  &\scriptstyle 0  &\scriptstyle 0  &\scriptstyle 0  &\scriptstyle 0  &\scriptstyle 0  &\scriptstyle 0  &\scriptstyle 0  &\scriptstyle 0  &\scriptstyle 0 &\scriptstyle 0 &\none &\none &\none &\none &\none &\none &\none &\none \cr 
        \none &\none &\none& \none & \none &\none &\scriptstyle 0  &\scriptstyle 0  &\scriptstyle 0  &\scriptstyle 0  &\scriptstyle 0  &\scriptstyle 0  &\scriptstyle 0  &\scriptstyle 0  &\scriptstyle 0  &\scriptstyle 0  &\scriptstyle 0  &\scriptstyle 0 &\scriptstyle 0 &\none &\none &\none &\none &\none &\none &\none \cr 
        \none &\none &\none &\none & \none & \none &\none &\scriptstyle 0  &\scriptstyle 0  &\scriptstyle 0  &\scriptstyle 0  &\scriptstyle 0  &\scriptstyle 0  &\scriptstyle 0  &\scriptstyle 0  &\scriptstyle 0  &\scriptstyle 0  &\scriptstyle 0  &\scriptstyle 0 &\scriptstyle 0 &\none &\none &\none &\none &\none &\none \cr 
        \none &\none &\none &\none &\none & \none & \none &\none &\scriptstyle 0  &\scriptstyle 0  &\scriptstyle 0  &\scriptstyle 0  &\scriptstyle 0  &\scriptstyle 0  &\scriptstyle 0  &\scriptstyle 0  &\scriptstyle 0  &\scriptstyle 0  &\scriptstyle 0  &\scriptstyle 0 &\scriptstyle 0 &\none &\none &\none &\none &\none \cr 
        \none &\none &\none &\none &\none &\none & \none & \none &\none  &\scriptstyle 0  &\scriptstyle 0  &\scriptstyle 0  &\scriptstyle 0  &\scriptstyle 0  &\scriptstyle 0  &\scriptstyle 0  &\scriptstyle 0  &\scriptstyle 0  &\scriptstyle 0  &\scriptstyle 0  &\scriptstyle 0 &\scriptstyle 0 &\none &\none &\none &\none \cr 
        \none &\none &\none &\none &\none &\none &\none & \none & \none &\none &\scriptstyle 0  &\scriptstyle 0  &\scriptstyle 0  &\scriptstyle 0  &\scriptstyle 0  &\scriptstyle 0  &\scriptstyle 0  &\scriptstyle 0  &\scriptstyle 0  &\scriptstyle 0  &\scriptstyle 0  &\scriptstyle 0 &\scriptstyle 0 &\none &\none &\none \cr 
        \end{ytableau}
    \end{adjustbox}

    As above, we use the Serre functor of $\Gc r(1, \Ec)$ to translate the rows horizontally. We find the following arrangement, where the yellow area, as usual, is for ease of notation in the next steps:

    \begin{adjustbox}{width=.5\textwidth,center}
        \ytableausetup{boxsize=2em}
        \begin{ytableau}
        \none & \none & \none & \none & \none & \none & *(lightgray) \cr 
        \scriptstyle  0 &\scriptstyle 0  &\scriptstyle 0  &\scriptstyle 0  &\scriptstyle 0    &\scriptstyle 0  &\scriptstyle 0  &\scriptstyle 0  &\scriptstyle 0  &\scriptstyle 0  &\scriptstyle 0  &\scriptstyle 0  &\scriptstyle 0\cr 
        \scriptstyle  0 &*(yellow)\scriptstyle 0  &*(yellow)\scriptstyle 0  &*(yellow)\scriptstyle 0  &*(yellow)\scriptstyle 0  &*(yellow)\scriptstyle 0  &\scriptstyle 0  &\scriptstyle 0 &\scriptstyle 0  &\scriptstyle 0  &\scriptstyle 0  &\scriptstyle 0   &\scriptstyle 0 \cr 
        \scriptstyle  0 &*(yellow)\scriptstyle 0  &*(yellow)\scriptstyle 0  &*(yellow)\scriptstyle 0  &*(yellow)\scriptstyle 0  &*(yellow)\scriptstyle 0  &\scriptstyle 0  &\scriptstyle 0  &\scriptstyle 0  &\scriptstyle 0  &\scriptstyle 0    &\scriptstyle 0  &\scriptstyle 0\cr 
        \scriptstyle  0 &*(yellow)\scriptstyle 0  &*(yellow)\scriptstyle 0  &*(yellow)\scriptstyle 0  &*(yellow)\scriptstyle 0  &*(yellow)\scriptstyle 0  &\scriptstyle 0  &\scriptstyle 0  &\scriptstyle 0  &\scriptstyle 0    &\scriptstyle 0  &\scriptstyle 0  &\scriptstyle 0\cr 
        \scriptstyle  0 &*(yellow)\scriptstyle 0  &*(yellow)\scriptstyle 0    &*(yellow)\scriptstyle 0  &*(yellow)\scriptstyle 0  &*(yellow)\scriptstyle 0  &\scriptstyle 0    &\scriptstyle 0  &\scriptstyle 0  &\scriptstyle 0  &\scriptstyle 0  &\scriptstyle 0&\scriptstyle 0\cr 
        \scriptstyle  0 &*(yellow)\scriptstyle 0  &*(yellow)\scriptstyle 0  &*(yellow)\scriptstyle 0  &*(yellow)\scriptstyle 0  &*(yellow)\scriptstyle 0  &\scriptstyle 0  &\scriptstyle 0    &\scriptstyle 0  &\scriptstyle 0  &\scriptstyle 0  &\scriptstyle 0  &\scriptstyle 0\cr 
        \scriptstyle  0 &*(yellow)\scriptstyle 0  &*(yellow)\scriptstyle 0  &*(yellow)\scriptstyle 0  &*(yellow)\scriptstyle 0  &*(yellow)\scriptstyle 0  &\scriptstyle 0    &\scriptstyle 0  &\scriptstyle 0  &\scriptstyle 0  &\scriptstyle 0  &\scriptstyle 0  &\scriptstyle 0\cr 
        \scriptstyle  0 &*(yellow)\scriptstyle 0  &*(yellow)\scriptstyle 0  &*(yellow)\scriptstyle 0  &*(yellow)\scriptstyle 0    &\scriptstyle 0  &\scriptstyle 0  &\scriptstyle 0  &\scriptstyle 0  &\scriptstyle 0  &\scriptstyle 0  &\scriptstyle 0  &\scriptstyle 0\cr 
        \scriptstyle  0 &*(yellow)\scriptstyle 0  &*(yellow)\scriptstyle 0  &*(yellow)\scriptstyle 0  &\scriptstyle 0    &\scriptstyle 0  &\scriptstyle 0  &\scriptstyle 0  &\scriptstyle 0  &\scriptstyle 0  &\scriptstyle 0  &\scriptstyle 0  &\scriptstyle 0\cr 
        \scriptstyle  0 &*(yellow)\scriptstyle 0  &*(yellow)\scriptstyle 0  &\scriptstyle 0  &\scriptstyle 0    &\scriptstyle 0  &\scriptstyle 0  &\scriptstyle 0  &\scriptstyle 0  &\scriptstyle 0  &\scriptstyle 0  &\scriptstyle 0  &\scriptstyle 0\cr 
        \scriptstyle  0 &*(yellow)\scriptstyle 0  &\scriptstyle 0  &\scriptstyle 0  &\scriptstyle 0    &\scriptstyle 0  &\scriptstyle 0  &\scriptstyle 0  &\scriptstyle 0  &\scriptstyle 0  &\scriptstyle 0  &\scriptstyle 0  &\scriptstyle 0\cr 
        \end{ytableau}
    \end{adjustbox}

    \subsubsection{First upward phase}
    We essentially proceed as in the even case, but with a different yellow area. We mutate away all the object in such area by mutating the first line bundle (in the leftmost white column) through them, and then by moving them through the row immediately above until they get canceled by the rule of Lemma \ref{lem:rule1}. The outcome is the following:
    
    \begin{adjustbox}{width=.5\textwidth,center}
        \ytableausetup{boxsize=2em}
        \begin{ytableau}
        \none & \none & \none & \none & \none & \none & *(lightgray) \cr 
        \scriptstyle  0 &\scriptstyle 0  &\scriptstyle 1  &\scriptstyle 2  &\scriptstyle 3    &\scriptstyle 4  &\scriptstyle 5  &\scriptstyle 0  &\scriptstyle 0  &\scriptstyle 0  &\scriptstyle 0  &\scriptstyle 0  &\scriptstyle 0\cr 
        \scriptstyle  \times & \none   & \none   & \none   & \none   & \none   &\scriptstyle 5  &\scriptstyle 0 &\scriptstyle 0  &\scriptstyle 0  &\scriptstyle 0  &\scriptstyle 0   &\scriptstyle 0 \cr 
        \scriptstyle  \times & \none   & \none   & \none   & \none   & \none   &\scriptstyle 5  &\scriptstyle 0  &\scriptstyle 0  &\scriptstyle 0  &\scriptstyle 0    &\scriptstyle 0  &\scriptstyle 0\cr 
        \scriptstyle  \times & \none   & \none   & \none   & \none   & \none   &\scriptstyle 5  &\scriptstyle 0  &\scriptstyle 0  &\scriptstyle 0    &\scriptstyle 0  &\scriptstyle 0  &\scriptstyle 0\cr 
        \scriptstyle  \times & \none   & \none     & \none   & \none   & \none   &\scriptstyle 5    &\scriptstyle 0  &\scriptstyle 0  &\scriptstyle 0  &\scriptstyle 0  &\scriptstyle 0&\scriptstyle 0\cr 
        \scriptstyle  \times & \none   & \none   & \none   & \none   & \none   &\scriptstyle 5  &\scriptstyle 0    &\scriptstyle 0  &\scriptstyle 0  &\scriptstyle 0  &\scriptstyle 0  &\scriptstyle 0\cr 
        \scriptstyle  \times & \none   & \none   & \none   & \none     & \none   &\scriptstyle 0  &\scriptstyle 0  &\scriptstyle 0  &\scriptstyle 0  &\scriptstyle 0  &\scriptstyle 0  &\scriptstyle 0\cr 
        \scriptstyle  \times & \none   & \none   & \none   & \none     & \scriptstyle 0  & \scriptstyle 0  &\scriptstyle 0  &\scriptstyle 0  &\scriptstyle 0  &\scriptstyle 0  &\scriptstyle 0  &\scriptstyle 0\cr 
        \scriptstyle  \times & \none   & \none   & \none   & \scriptstyle 0    & \scriptstyle 0  & \scriptstyle 0  &\scriptstyle 0  &\scriptstyle 0  &\scriptstyle 0  &\scriptstyle 0  &\scriptstyle 0  &\scriptstyle 0\cr 
        \scriptstyle  \times & \none   & \none   & \scriptstyle 0  & \scriptstyle 0    & \scriptstyle 0  & \scriptstyle 0  &\scriptstyle 0  &\scriptstyle 0  &\scriptstyle 0  &\scriptstyle 0  &\scriptstyle 0  &\scriptstyle 0\cr 
        \scriptstyle  \times & \none   & \scriptstyle 0  & \scriptstyle 0  & \scriptstyle 0    & \scriptstyle 0  & \scriptstyle 0  &\scriptstyle 0  & \scriptstyle 0  &\scriptstyle 0  &\scriptstyle 0  &\scriptstyle 0  &\scriptstyle 0\cr 
        \end{ytableau}
    \end{adjustbox}

    Let us now apply the Serre functor to the first $n$ blocks of the first row, and let us identify a new yellow area for the next upward phase:
    
    \begin{adjustbox}{width=.5\textwidth,center}
        \ytableausetup{boxsize=2em}
        \begin{ytableau}
        \none & \none & \none & \none & \none & \none & *(lightgray) \cr 
        \none & \none & \none & \none & \none & \none &\scriptstyle 5  &\scriptstyle 0  &\scriptstyle 0  &\scriptstyle 0  &\scriptstyle 0  &\scriptstyle 0  &\scriptstyle 0\cr 
        \scriptstyle  \times & \none   & \none   & \none   & \none   & \none   &\scriptstyle 5  &\scriptstyle 0 &\scriptstyle 0  &\scriptstyle 0  &\scriptstyle 0  &\scriptstyle 0   &\scriptstyle 0 \cr 
        \scriptstyle  \times & \none   & \none   & \none   & \none   & \none   &\scriptstyle 5  &\scriptstyle 0  &\scriptstyle 0  &\scriptstyle 0  &\scriptstyle 0    &\scriptstyle 0  &\scriptstyle 0\cr 
        \scriptstyle  \times & \none   & \none   & \none   & \none   & \none   &\scriptstyle 5  &\scriptstyle 0  &\scriptstyle 0  &\scriptstyle 0    &\scriptstyle 0  &\scriptstyle 0  &\scriptstyle 0\cr 
        \scriptstyle  \times & \none   & \none     & \none   & \none   & \none   &\scriptstyle 5    &\scriptstyle 0  &\scriptstyle 0  &\scriptstyle 0  &\scriptstyle 0  &\scriptstyle 0&\scriptstyle 0\cr 
        \scriptstyle  \times & \none   & \none   & \none   & \none   & \none   &\scriptstyle 5  &\scriptstyle 0    &\scriptstyle 0  &\scriptstyle 0  &\scriptstyle 0  &\scriptstyle 0  &\scriptstyle 0\cr 
        \scriptstyle  \times & \none   & \none   & \none   & \none     & \none   &\scriptstyle 0  &\scriptstyle 0  &\scriptstyle 0  &\scriptstyle 0  &\scriptstyle 0  &\scriptstyle 0  &\scriptstyle 0\cr 
        \scriptstyle  \times & \none   & \none   & \none   & \none     &*(yellow)\scriptstyle 0  &*(yellow)\scriptstyle 0  &\scriptstyle 0  &\scriptstyle 0  &\scriptstyle 0  &\scriptstyle 0  &\scriptstyle 0  &\scriptstyle 0\cr 
        \scriptstyle  \times & \none   & \none   & \none   &*(yellow)\scriptstyle 0    &*(yellow)\scriptstyle 0  &*(yellow)\scriptstyle 0  &\scriptstyle 0  &\scriptstyle 0  &\scriptstyle 0  &\scriptstyle 0  &\scriptstyle 0  &\scriptstyle 0\cr 
        \scriptstyle  \times & \none   & \none   &*(yellow)\scriptstyle 0  &*(yellow)\scriptstyle 0    &*(yellow)\scriptstyle 0  &*(yellow)\scriptstyle 0  &\scriptstyle 0  &\scriptstyle 0  &\scriptstyle 0  &\scriptstyle 0  &\scriptstyle 0  &\scriptstyle 0\cr 
        \scriptstyle  \times & \none   &*(yellow)\scriptstyle 0  &*(yellow)\scriptstyle 0  &*(yellow)\scriptstyle 0    &*(yellow)\scriptstyle 0  &*(yellow)\scriptstyle 0  &\scriptstyle 0  & \scriptstyle 0  &\scriptstyle 0  &\scriptstyle 0  &\scriptstyle 0  &\scriptstyle 0\cr 
        \none & *(yellow)\scriptstyle  0 &*(yellow)\scriptstyle 0  &*(yellow)\scriptstyle 1  &*(yellow)\scriptstyle 2  &*(yellow)\scriptstyle 3    &*(yellow)\scriptstyle 4 
        \end{ytableau}
    \end{adjustbox}

    \subsubsection{Second upward phase}
    We apply for the last time the rule of Lemma \ref{lem:rule1}, getting rid of the yellow boxes once again:

    \begin{adjustbox}{width=.5\textwidth,center}
        \ytableausetup{boxsize=2em}
        \begin{ytableau}
        \none & \none & \none & \none & \none & \none & *(lightgray) \cr 
        \none & \none & \none & \none & \none & \none &\scriptstyle 5  &\scriptstyle 0  &\scriptstyle 0  &\scriptstyle 0  &\scriptstyle 0  &\scriptstyle 0  &\scriptstyle 0\cr 
        \scriptstyle  \times & \none   & \none   & \none   & \none   & \none   &\scriptstyle 5  &\scriptstyle 0 &\scriptstyle 0  &\scriptstyle 0  &\scriptstyle 0  &\scriptstyle 0   &\scriptstyle 0 \cr 
        \scriptstyle  \times & \none   & \none   & \none   & \none   & \none   &\scriptstyle 5  &\scriptstyle 0  &\scriptstyle 0  &\scriptstyle 0  &\scriptstyle 0    &\scriptstyle 0  &\scriptstyle 0\cr 
        \scriptstyle  \times & \none   & \none   & \none   & \none   & \none   &\scriptstyle 5  &\scriptstyle 0  &\scriptstyle 0  &\scriptstyle 0    &\scriptstyle 0  &\scriptstyle 0  &\scriptstyle 0\cr 
        \scriptstyle  \times & \none   & \none     & \none   & \none   & \none   &\scriptstyle 5    &\scriptstyle 0  &\scriptstyle 0  &\scriptstyle 0  &\scriptstyle 0  &\scriptstyle 0&\scriptstyle 0\cr 
        \scriptstyle  \times & \none   & \none   & \none   & \none   & \none   &\scriptstyle 5  &\scriptstyle 0    &\scriptstyle 0  &\scriptstyle 0  &\scriptstyle 0  &\scriptstyle 0  &\scriptstyle 0\cr 
        \scriptstyle  \times & \none   & \none   & \none   & \none     & \none   &\scriptstyle 5  &\scriptstyle 5  &\scriptstyle 0  &\scriptstyle 0  &\scriptstyle 0  &\scriptstyle 0  &\scriptstyle 0\cr 
        \scriptstyle  \times & \none   & \none   & \none   & \none     &\none    &\none    &\scriptstyle 5  &\scriptstyle 0  &\scriptstyle 0  &\scriptstyle 0  &\scriptstyle 0  &\scriptstyle 0\cr 
        \scriptstyle  \times & \none   & \none   & \none   &\none      &\none    &\none    &\scriptstyle 5  &\scriptstyle 0  &\scriptstyle 0  &\scriptstyle 0  &\scriptstyle 0  &\scriptstyle 0\cr 
        \scriptstyle  \times & \none   & \none   &\none    &\none      &\none    &\none    &\scriptstyle 5  &\scriptstyle 0  &\scriptstyle 0  &\scriptstyle 0  &\scriptstyle 0  &\scriptstyle 0\cr 
        \scriptstyle  \times & \none   &\none    &\none    &\none      &\none    &\none    &\scriptstyle 5  & \scriptstyle 0  &\scriptstyle 0  &\scriptstyle 0  &\scriptstyle 0  &\scriptstyle 0\cr 
        \none & \none &\none    & \none  & \none  & \none   & \none 
        \end{ytableau}
    \end{adjustbox}

    Let us now move the second column of ``$n-1$'' blocks to the end of the collection (by mutating all the zeros at its right), and then let us send it to the beginning via Serre functor. We find:

    \begin{adjustbox}{width=.5\textwidth,center}
        \ytableausetup{boxsize=2em}
        \begin{ytableau}
        \none & \none & \none & \none & \none & \none & \scriptstyle 5 \cr 
        \none & \none & \none & \none & \none & \none & \scriptstyle 5 \cr 
        \none & \none & \none & \none & \none & \none & \scriptstyle 5 \cr 
        \none & \none & \none & \none & \none & \none & \scriptstyle 5 \cr 
        \none & \none & \none & \none & \none & \none & *(lightgray) \scriptstyle 5\cr 
        \none & \none & \none & \none & \none & \none &\scriptstyle 5  &\scriptstyle 0  &\scriptstyle 0  &\scriptstyle 0  &\scriptstyle 0  &\scriptstyle 0  &\scriptstyle 0\cr 
        \scriptstyle  \times & \none   & \none   & \none   & \none   & \none   &\scriptstyle 5  &\scriptstyle 0 &\scriptstyle 0  &\scriptstyle 0  &\scriptstyle 0  &\scriptstyle 0   &\scriptstyle 0 \cr 
        \scriptstyle  \times & \none   & \none   & \none   & \none   & \none   &\scriptstyle 5  &\scriptstyle 0  &\scriptstyle 0  &\scriptstyle 0  &\scriptstyle 0    &\scriptstyle 0  &\scriptstyle 0\cr 
        \scriptstyle  \times & \none   & \none   & \none   & \none   & \none   &\scriptstyle 5  &\scriptstyle 0  &\scriptstyle 0  &\scriptstyle 0    &\scriptstyle 0  &\scriptstyle 0  &\scriptstyle 0\cr 
        \scriptstyle  \times & \none   & \none     & \none   & \none   & \none   &\scriptstyle 5    &\scriptstyle 0  &\scriptstyle 0  &\scriptstyle 0  &\scriptstyle 0  &\scriptstyle 0&\scriptstyle 0\cr 
        \scriptstyle  \times & \none   & \none   & \none   & \none   & \none   &\scriptstyle 5  &\scriptstyle 0    &\scriptstyle 0  &\scriptstyle 0  &\scriptstyle 0  &\scriptstyle 0  &\scriptstyle 0\cr 
        \scriptstyle  \times & \none   & \none   & \none   & \none     & \none   &\scriptstyle 5  & \none &\scriptstyle \times  &\scriptstyle \times  &\scriptstyle \times  &\scriptstyle \times  &\scriptstyle \times\cr 
        \scriptstyle  \times & \none   & \none   & \none   & \none     &\none    &\none    & \none  &\scriptstyle \times  &\scriptstyle \times  &\scriptstyle \times  &\scriptstyle \times  &\scriptstyle \times\cr 
        \scriptstyle  \times & \none   & \none   & \none   &\none      &\none    &\none    &\none  &\scriptstyle \times  &\scriptstyle \times  &\scriptstyle \times  &\scriptstyle \times  &\scriptstyle \times\cr 
        \scriptstyle  \times & \none   & \none   &\none    &\none      &\none    &\none    &\none  &\scriptstyle \times  &\scriptstyle \times  &\scriptstyle \times  &\scriptstyle \times  &\scriptstyle \times\cr 
        \scriptstyle  \times & \none   &\none    &\none    &\none      &\none    &\none    &\none  & \scriptstyle \times  &\scriptstyle \times &\scriptstyle \times  &\scriptstyle \times  &\scriptstyle \times \cr 
        \end{ytableau}
    \end{adjustbox}

    As for the even case, the last step consists in applying Corollary \ref{cor:rule2} to the last ``$n-1$'' block, and sending the new block to the beginning via Serre functor. Hence, we produce the final chessboard:

    \begin{adjustbox}{width=.5\textwidth,center}
        \ytableausetup{boxsize=2em}
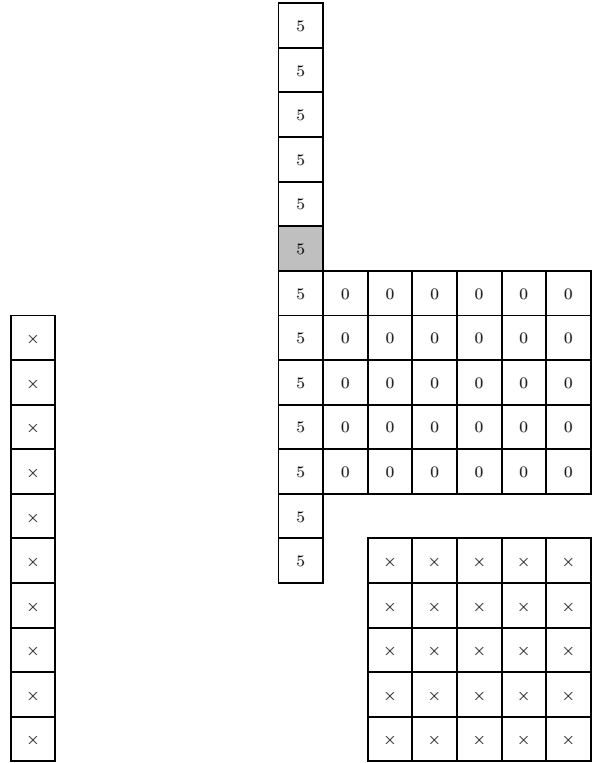
        \begin{ytableau}
        \none & \none & \none & \none & \none & \none & \scriptstyle 5 \cr 
        \none & \none & \none & \none & \none & \none & \scriptstyle 5 \cr 
        \none & \none & \none & \none & \none & \none & \scriptstyle 5 \cr 
        \none & \none & \none & \none & \none & \none & \scriptstyle 5 \cr 
        \none & \none & \none & \none & \none & \none & \scriptstyle 5 \cr 
        \none & \none & \none & \none & \none & \none & *(lightgray) \scriptstyle 5\cr 
        \none & \none & \none & \none & \none & \none &\scriptstyle 5  &\scriptstyle 0  &\scriptstyle 0  &\scriptstyle 0  &\scriptstyle 0  &\scriptstyle 0  &\scriptstyle 0\cr 
        \scriptstyle  \times & \none   & \none   & \none   & \none   & \none   &\scriptstyle 5  &\scriptstyle 0 &\scriptstyle 0  &\scriptstyle 0  &\scriptstyle 0  &\scriptstyle 0   &\scriptstyle 0 \cr 
        \scriptstyle  \times & \none   & \none   & \none   & \none   & \none   &\scriptstyle 5  &\scriptstyle 0  &\scriptstyle 0  &\scriptstyle 0  &\scriptstyle 0    &\scriptstyle 0  &\scriptstyle 0\cr 
        \scriptstyle  \times & \none   & \none   & \none   & \none   & \none   &\scriptstyle 5  &\scriptstyle 0  &\scriptstyle 0  &\scriptstyle 0    &\scriptstyle 0  &\scriptstyle 0  &\scriptstyle 0\cr 
        \scriptstyle  \times & \none   & \none     & \none   & \none   & \none   &\scriptstyle 5    &\scriptstyle 0  &\scriptstyle 0  &\scriptstyle 0  &\scriptstyle 0  &\scriptstyle 0&\scriptstyle 0\cr 
        \scriptstyle  \times & \none   & \none   & \none   & \none   & \none   &\scriptstyle 5  \cr 
        \scriptstyle  \times & \none   & \none   & \none   & \none     & \none   &\scriptstyle 5  & \none &\scriptstyle \times  &\scriptstyle \times  &\scriptstyle \times  &\scriptstyle \times  &\scriptstyle \times\cr 
        \scriptstyle  \times & \none   & \none   & \none   & \none     &\none    &\none    & \none  &\scriptstyle \times  &\scriptstyle \times  &\scriptstyle \times  &\scriptstyle \times  &\scriptstyle \times\cr 
        \scriptstyle  \times & \none   & \none   & \none   &\none      &\none    &\none    &\none  &\scriptstyle \times  &\scriptstyle \times  &\scriptstyle \times  &\scriptstyle \times  &\scriptstyle \times\cr 
        \scriptstyle  \times & \none   & \none   &\none    &\none      &\none    &\none    &\none  &\scriptstyle \times  &\scriptstyle \times  &\scriptstyle \times  &\scriptstyle \times  &\scriptstyle \times\cr 
        \scriptstyle  \times & \none   &\none    &\none    &\none      &\none    &\none    &\none  & \scriptstyle \times  &\scriptstyle \times &\scriptstyle \times  &\scriptstyle \times  &\scriptstyle \times \cr 
        \end{ytableau}
    \end{adjustbox}
    
    Here we finally recognize all the blocks composing $\dbcoh(\Gc r(2, \Ec))$, although they need to be mutated to the end of the collection. By the same exact approach of the previous section we prove the following:
    \begin{proposition}\label{prop:main_theorem_fibrations_odd}
        Let $X_1$ and $X_2$ be as described in Section \ref{sec:cover_and_fano_fibration}, for $\epsilon = 1$. Then there is a fully faithful functor $\Psi:\dbcoh(X_1)\xhookrightarrow{\,\,\,\,\,}\dbcoh(X_2)$. Moreover, $\dbcoh(X_2)$ admits a semiorthogonal decomposition as follows:
        \begin{equation*}
            \dbcoh(X_2) = \langle \Psi(\dbcoh(X_1)), F_1, \dots, F_{2n^2-n-1}\rangle
        \end{equation*}
        where, for all $i$, $F_i$ is an admissible subcategory given by the image of $\dbcoh(B)$ through a fully faithful functor.
    \end{proposition}
    Together with Proposition \ref{prop:main_theorem_fibrations_even}, Proposition \ref{prop:main_theorem_fibrations_odd} completes the proof of Theorem \ref{thm:main_theorem_fibrations_intro}.
    \begin{remark}
        Note that Theorem \ref{thm:main_theorem_fibrations_intro} holds without assuming Conditions \ref{cond:restriction_surjective} and \ref{cond:dimension_of_B}. In fact, the sole Condition \ref{cond:bpf} is necessary to have smoothness of $X_1$, $X_2$ and $\Mc$, which, in turn, allows to write the semiorthogonal decompositions of Section \ref{sec:sods_for_M} and proceed with the chess game.
    \end{remark}
    \subsection{A note on tilting bundles}
        In the paper \cite{ourpaper_generalizedroofs}, the case $B = \{pt\}$, which had previously  been addressed by \cite{leung_xie}, has been revisited with a different approach, based on GLSM phase transitions and the construction of a window category. The approach, inspired by \cite{addingtondonovansegal}, is the following:
        \begin{enumerate}
            \item observe that $G(1, 2n+\epsilon)$, $G(2, 2n+\epsilon)$, and the total spaces $X_+$ and $X_-$ respectively of $\Qc_{G(1, 2n+\epsilon)}(-2)$ and of $\Uc_{G(2, 2n+\epsilon)}^\vee(-2)$ are GIT quotients, and that $X_+$ and $X_-$  are both birartional to an Artin stack $[V/GL(2n+\epsilon-2)]$ where $V$ is a vector space. There is a function $f:V\arw\CC$ such that the zero loci $Y_+$ and $Y_-$ of general secions of the duals of $\Qc_{G(1, 2n+\epsilon)}(-2)$ and of $\Uc_{G(2, 2n+\epsilon)}^\vee(-2)$ can be realized as the critical loci of $f$ restricted to $X_+$ and $X_-$
            \item There is a tilting bundle of $X_+$ which, under the birational map $X_+\arw X_-$, is sent to a partially tilting bundle
            \item By means of Kn\"orrer periodicity \cite{shipman} and passing to derived category of matrix factorizations, one has a composition of functors:
            \begin{equation*}
                \dbcoh(Y_-)\simeq\operatorname{DMF}(X_-, f)\xhookrightarrow{\,\,\,\,\,}\operatorname{DMF}(X_+, f)\simeq\dbcoh(Y_+).
            \end{equation*}
        \end{enumerate}
        This technique, in principle, could be used for the present setting, giving a simpler and shorter proof of Theorem \ref{thm:main_theorem_fibrations_intro}, which is also more suitable to generalization to the more complicated case of $\Gc(k, k+1, \Ec)$. However, a tilting bundle on $G(1, 2n+\epsilon)$ does not directlty imply the existence of a tilting bundle on a Grassmann bundle with fiber $G(1, 2n+\epsilon)$ on any smooth base: therefore, we opted for the argument presented in Section \ref{sec:derived_cats_cover_and_fibration}.
\section{Examples}\label{sec:examples}
    Let us discuss some concrete examples. We address the case where $B, \Gc r(i, \Ec)$ and $\Fc l(1,2,\Ec)$ are flag varieties themselves: the outcome is an infinite series of derived embeddings between covers of $B$ and Fano fibrations over $B$, where both the cover and the fibration are cut by general sections of irreducible, homogeneous vector bundles. More precisely, fix $n\leq 2$ and any strictly increasing partition $\mu = \{\mu_1, \dots, \mu_r\}$ of arbitrary length $r$. Then, we choose the data of $\Ec\arw B$ so that $\Fc l (1, 2, \Ec) = F(\mu_1, \dots, \mu_r, \mu_r + 1, \mu_r + 2, \mu_r + 2n+\epsilon)$. Such variety has exactly $r+2$ extremal contractions to flag varieties of Picard rank $r+1$, which in turn have $r+1$ extremal contractions to flag varieties of Picard rank $r$. This process can obviously be iterated, and eventually defines locally trivial morphisms
    \begin{equation*}
        \begin{tikzcd}[row sep = tiny, column sep = large, /tikz/column 1/.append style={anchor=base east} ,/tikz/column 2/.append style={anchor=base west}]
            F(\mu_1, \dots, \mu_r, \mu_r + 1, \mu_r + 2, \mu_r + 2n+\epsilon) \ar{r}{\phi_j} &  G(j, \mu_r + 2n+\epsilon)\\
            (x_1, \dots, x_{r+2}) \ar[maps to]{r} & x_j
        \end{tikzcd}
    \end{equation*}
    where $j\in\{\mu_1, \dots, \mu_r, \mu_r+1, \mu_r+2\}$. Fibers of these maps are products of flag varieties. We call $\Uc_j$ the pullback of the tautological bundle through $\phi_j$, a vector bundle of rank $j$, and we denote by $\Qc_j$ the rank $\mu_r + 2n+\epsilon-j$ vector bundle defined as the quotient of $\Oc^{\oplus(2n+\epsilon)}$ by the latter, via the tautological embedding. Let us use the notation $\Oc(-h_j):=\det\Uc_j$ for line bundles. We will also use the shorthand notation $F(\mu, \mu_r + 1, \mu_r + 2, \mu_r + 2n+\epsilon)$ for $F(\mu_1, \dots, \mu_r, \mu_r + 1, \mu_r + 2, \mu_r + 2n+\epsilon)$. In this context, the diagram \ref{eq:main_diagram} specializes to:
    \begin{equation}\label{eq:main_diagram_examples}
        \begin{tikzcd}[row sep = huge, column sep = scriptsize]
            & F(\mu, \mu_r + 1, \mu_r + 2, \mu_r+2n+\epsilon)\ar[swap]{dl}{p_1}\ar{dr}{p_2} &  \\
            F(\mu, \mu_r + 1, \mu_r+2n+\epsilon)\ar[swap]{dr}{r_1} & & F(\mu, \mu_r + 2, \mu_r+2n+\epsilon)\ar{dl}{r_2} \\
            &F(\mu, \mu_r+2n+\epsilon)& 
        \end{tikzcd}
    \end{equation}
    Given any dominant \footnote{ the dominant condition is required in order to have $\Lc$ satisfy Condition \ref{cond:bpf}, which, in turn, is necessary to have smoothness of $X_1$ and $X_2$.} weight $\nu = \nu_1\omega_1+\dots+\nu_r\omega_r$, we consider the line bundle $\Lc := \Ec_{\nu+\omega_{\mu_r+1}+\omega_{\mu_r+2}}$. Then, one has:
    \begin{equation*}
        \begin{split}
            p_{1*}\Lc &\simeq \Ec_\nu\otimes \wedge^{\mu_r+2n+\epsilon-\mu_r-2}\Qc_{\mu_r+1}(h_{\mu_r+1})\simeq \Ec_\nu\otimes \Qc_{\mu_r+1}^\vee(2h_{\mu_r+1})\\
            p_{2*}\Lc &\simeq \Ec_\nu\otimes \Pc(2h_{\mu_r + 2})
        \end{split}
    \end{equation*}
    where $\Pc(h_{\mu_r+2})$ is defined as (the pullback of) the homogeneous, irreducible, globally generated vector bundle of highest weight $\omega_{\mu_r+1}$ on $F(\mu_r, \mu_r+2, \mu_r+2n+\epsilon)$. Note that one has:
    \begin{equation*}
        0\arw\Uc_{\mu_r}(h_{\mu_r+2})\arw\Uc_{\mu_r+2}(h_{\mu_r+2})\arw\Pc(h_{\mu_r + 2})\arw 0.
    \end{equation*}
    By adjunction, one easily sees that for every $x\in F(\mu_1, \dots, \mu_r, \mu_r+2n+\epsilon)$:
    \begin{equation*}
        \omega_{X_1}|_{r_1^{-1}(x)} = \Oc(2n+\epsilon-3) \simeq \omega_{X_2}|_{r_2^{-1}(x)}^\vee,
    \end{equation*}
    and therefore, for our assumptions on $\mu, r$ and $n$, we see that the general fiber of $X_2$ is a smooth Fano variety of index $2n+\epsilon-3$.
    \subsection{The case of \texorpdfstring{$n=2, \epsilon = 1$}{something}}
        Let us briefly review the simplest case, which corresponds to considering a flag bundle with fiber isomorphic to $F(1, 2, 5)$.
        \subsubsection{A fibration in Fano fourfolds of degree $12$}
            With this data, for any $\mu$ we produce a fibration $X_2\arw F(\mu_1, \dots, \mu_r, \mu_r +5)$ with general fiber isomorphic to a smooth Fano fourfold $Y$ of index $2$, degree $12$ and genus $7$ in $G(2, 5)$, cut by a section of $\Uc_{G(2, 5)}^\vee(1)$. This variety is the case $14$ in the classification \cite[Table 6.5]{iskovskikh}, and it is usually described as a codimension $6$ linear section of the spinor tenfold (i.e a connected component of the orthogonal Grassmannian $OG(5, 10)$): the fact that the description we use is equivalent is well known, and it can be easily deduced by the argument presented in \cite[Section 13]{coates_corti_galkin_kasprzyk}. In particular, the argument there presented relates codimension $7$ general hyperplane sections of a connected component of $OG(5, 10)$ to general sections of $\Uc_{G(2, 5)}^\vee(1)\oplus\Oc(1)$, and it can be adapted to our case verbatim. By Lemma \ref{lem:the_cover_of_B} we immediately see that for $\dim(F(\mu, \mu_r+5))<10$  the map $X_1\arw F(\mu, \mu_r +5)$ is a $11:1$ cover. In particular, by choosing $\mu = \{1\}$, we get a cover of $\PP^5$. Already by choosing $\mu = \{2\}$ we have a generically $11:1$ morphism to $G(2, 7)$, with positive dimensional fibers over a zero-dimensional subset, by all other choice of $\mu$ there is a positive dimensional subvariety over which the morphism is not finite.
        \subsubsection{Relation with Kuznetsov's collection for $Y$}
            By Theorem \ref{thm:main_theorem_fibrations_intro},
            there is a fully faithful functor $\dbcoh(X_1)\subset\dbcoh(X_2)$ and a semiorthogonal decomposition:
            \begin{equation*}
                \dbcoh(X_2) = \langle \Psi\dbcoh(X_1), F_1, \dots, F_5\rangle
            \end{equation*}
            A full exceptional collection for the general fiber of $Y$ of $X_2$ can be found if we consider the case $B = \{pt\}$: 
            \begin{equation}\label{collection_fano_4_degree_12}
                \dbcoh(Y) = \langle P_1, \dots, P_{11}, F_1, \dots F_5\rangle
            \end{equation}
            where $P_1, \dots P_{11}$ are the objects coming from the derived category of a set of $11$ distinct points. In \cite{kuznetsov_spinor} a different full exceptional collection for $Y$ has been produced, using the fact that $Y$ is a linear section of the spinor tenfold (which admits a Lefschetz, rectangular full exceptional collection of vector bundles): if we call $\Uc_+$ the tautological bundle of the spinor tenfold (the pullback of the tautological bundle of $G(5, 10)$), one has:
            \begin{equation*}
                \dbcoh(Y) = \langle E_1, \dots, E_{12}, \Oc, \Uc_+^\vee, \Oc(1), \Uc_+^\vee(1)\rangle.
            \end{equation*}
            where the $E_i$'s generate the derived category of a set of $12$ distinct points (a codimension 10 linear section of the homological projective dual of the spinor tenfold, where the latter, remarkably, is isomorphic to the spinor tenfold itself). Note that both collections have length 16. It would be interesting to understand whether the two collections can be related by a sequence of mutations, or, more generally, if there is any geometric relation between the collections.
            \begin{remark}
                Note that the embedding of categories $\langle P_1, \dots, P_{11}\rangle \subset \dbcoh(Y)$ is already a consequence of \cite{ourpaper_generalizedroofs}. However, proving that the right orthogonal complement of $\langle P_1, \dots, P_{11}\rangle$ is itself generated by an exceptional collection (and thus obtaining the collection \ref{collection_fano_4_degree_12}) is a consequence of the more explicit approach we discussed in Section \ref{sec:derived_cats_cover_and_fibration}.
            \end{remark}
\section{Grassmann flips on a base}\label{sec:grassmann_flip}
    \subsection{Simple flips and flag bundles}
        Consider a birational map between two smooth, projective varieties $\Xc_1$ and $\Xc_2$, resolved by two blowups $\pi_1:\Xc\arw \Xc_1$ and $\pi_2:\Xc\arw\Xc_2$. This is an instance of a \emph{simple flip} as described in \cite[Definition 2.2]{leung_xie_new}. Let us focus on the situation where the centers of the blowups are respectively isomorphic to the Grassmann bundles $\Gc r(1, \Ec)$ and $\Gc r(2, \Ec)$ for a suitable choice of a vector bundle $\Ec$ on a base $B$. The geometry described in Section \ref{sec:geoemtry_part_1}, together with \cite[Proposition 2.3]{leung_xie_new}, allows to draw the following diagram (cf. \cite[Diagram 2.4]{leung_xie_new}):
        \begin{equation}\label{eq_keqdiagram}
            \begin{tikzcd}[row sep=large, column sep = scriptsize]
                &   & \Fc l(1, 2, \Ec)\ar[hook]{d}{\sigma}\ar[swap]{lldd}{p_1}\ar{rrdd}{p_2} &  &\\  
                &   & \Xc \ar[swap]{ld}{\pi_1}\ar{rd}{\pi_2} & &\\
                \Gc r(1, \Ec)\ar[hook]{r}\ar[swap]{rrdd}{r_1}&   \Xc_1\ar[dashed]{rr}{\mu} & & \Xc_2\ar[hookleftarrow]{r} &\Gc r(2, \Ec)\ar{lldd}{r_2}\\
                &&\color{white}{AAAA}&&\\
                &&B&&
            \end{tikzcd}
        \end{equation}
        In particular, this is an instance of a simple flip of homogeneous type, as described in \cite{leung_xie_new}. 
        Note that, for $B = \{pt\}$, we obtain the construction addressed by \cite{leung_xie}. In light of this, it is reasonable to expect a derived embedding $\dbcoh(\Xc_1)\subset\dbcoh(\Xc_2)$. The goal of this section is to produce such embedding.
    \subsection{Semiorthogonal decompositions for \texorpdfstring{$\Xc$}{something}}
        By applying a result of Orlov \cite{orlovblowup} on semiorthogonal decompositions of blowups, we can construct two different semiorthogonal decompositions for $\Xc $:
        \begin{equation}\label{eq:initial_sod_for_X0_from_left}
            \begin{split}
                \dbcoh(\Xc ) & \simeq \langle \sigma_*(p_1^*\dbcoh(\Gc r(1, \Ec)\otimes\Lc^{\otimes(-2n-\epsilon+3)}), \dots, \sigma_*(p_1^*\dbcoh(\Gc r(1, \Ec)\otimes\Lc^{\otimes(-1)}), \pi_1^*\dbcoh(\Xc_1) \rangle \\
                & \simeq \langle \sigma_*(p_2^*\dbcoh(\Gc r(2, \Ec)\otimes\Lc^{\otimes(-1)}), \pi_2^*\dbcoh(\Xc_2) \rangle
            \end{split}
        \end{equation}
        In light of the semiorthogonal decomposition \ref{eq:sod_Gr1E} we can rewrite the above as follows:
        \begin{equation}\label{eq:sod_from_left_flip}
            \begin{split}
                \dbcoh(\Xc ) & \simeq \langle \mathbf T^0_{-2n-\epsilon+3, -2n-\epsilon+3}, \dots, \mathbf T^0_{3, -2n-\epsilon+3}, \\
                & \hspace{40pt} \vdots\hspace{100pt}\vdots \\
                 & \hspace{30pt} \mathbf T^0_{-1, -1}, \dots\dots\dots\dots, \mathbf T^0_{2n+\epsilon-1, -1}, \pi_1^*\dbcoh(\Xc_1) \rangle.
            \end{split}
        \end{equation}
        Here we introduced the subcategories:
        \begin{equation*}
            \begin{split}
                \mathbf T^m_{x,y}:= & \sigma_*(\Sym^m\wt\Uc_2^\vee(xh_1+yh_2)\otimes \tau^*\dbcoh(B)) \\
                \mathbf B^m_{x,y}:= & \langle \mathbf T^0_{x,y}, \mathbf T^1_{x,y}, \dots, \mathbf T^m_{x,y}\rangle.
            \end{split}
        \end{equation*}
        where $\tau := p_1\circ r_1 = p_2\circ r_2$.
        \begin{lemma}\label{lem:chess_game_exts_flip}
            For $0\leq r\leq\min(t-2, n-1)$ and $r\neq t-1$ one has:
            \begin{equation*}
                \LL_{\mathbf T^0_{t, 0}}\mathbf T^r_{-1, 1} \simeq \mathbf T^r_{-1, 1}.
            \end{equation*}
            Moreover, for $r = t-1$:
            \begin{equation*}
                \LL_{\mathbf T^0_{t, 0}}\mathbf T^{t-1}_{-1, 1} \simeq \mathbf T^t_{0,0}.
            \end{equation*}
        \end{lemma}
        \begin{proof}
            As for the similar claim in Section \ref{sec:derived_cats_cover_and_fibration}, let us start by observing that any object in $\mathbf T^r_{-1, 1}$ has the form $\sigma_*(\Sym^r\Uc_2^\vee(-h_1+h_2)\otimes\tau^* E)$ for some $E\in\dbcoh(B)$, and its mutation through $\mathbf T^0_{t, 0}$is defined by the distinguished triangle
            \begin{equation*}\label{eq:mutation_flip_distinguished_triangle}
                gg^!\sigma_*(\Sym^r\Uc_2^\vee(-h_1+h_2)\otimes\tau^* E) \arw \sigma_*(\Sym^r\Uc_2^\vee(-h_1+h_2)\otimes\tau^* E) \arw \LL_{\mathbf T^0_{t, 0}}\sigma_*(\Sym^r\Uc_2^\vee(-h_1+h_2)),
            \end{equation*}
            where $g$ and $g^!$ denote the following adjoint pair of functors:
            \begin{equation*}
                \begin{split}
                    g: & E\longmapsto \sigma_*(E\otimes\Oc(th_1))\\
                    g^! : & R \longmapsto \tau_* R\Hc om_{\Fc l(1, 2, \Ec)}(\sigma_*\Oc(th_1), R)
                \end{split}
            \end{equation*}
            Therefore, we can compute the first term of the triangle \ref{eq:mutation_flip_distinguished_triangle} explicitly:
            \begin{equation*}
                \begin{split}
                    gg^!\sigma_*(\Sym^r\Uc_2^\vee(-h_1+h_2)\otimes\tau^* E) 
                    \simeq \hspace{150pt}\\
                    \simeq g\tau_*R\Hc om_{\Fc l(1, 2, \Ec)}(\sigma_*\Oc(th_1), \sigma_*(\Sym^r\Uc_2^\vee(-h_1+h_2)\otimes\tau^*E) \\
                    \simeq g\tau_*R\Hc om_{\Fc l(1, 2, \Ec)}(\sigma^*\sigma_*\Oc(th_1), (\Sym^r\Uc_2^\vee(-h_1+h_2)\otimes\tau^*E) \\
                    \simeq g\tau_*R\Hc om_{\Fc l(1, 2, \Ec)}(\sigma^*\sigma_*\Oc(th_1), (\Sym^r\Uc_2^\vee(-h_1+h_2)\otimes\tau^*E).
                \end{split}
            \end{equation*}
            Note that every fiber of this object has the form:
            \begin{equation*}
                \begin{split}
                    (\tau_*R\Hc om_{\Fc l(1, 2, \Ec)}(\sigma^*\sigma_*\Oc(th_1), (\Sym^r\Uc_2^\vee(-h_1+h_2)\otimes\tau^*E))_b\simeq \\
                    \simeq H^\bullet(\tau^{-1}(b), \tau_*R\Hc om_{\Fc l(1, 2, \Ec)}(\sigma^*\sigma_*\Oc(th_1), (\Sym^r\Uc_2^\vee(-h_1+h_2)\otimes\tau^*E))|_{\tau^{-1}(b)} \\
                    \simeq \Hom_{F(1, 2, \Ec_b)}(\Oc(t, 0), \Sym^r\Uc_{G(2, \Ec_b)^\vee}(-1, 1).
                \end{split}
            \end{equation*}
            Then, we conclude as in the proof of the Claim inside the proof of Lemma \ref{lem:rule1}.
        \end{proof}
        We can now state the main theorem of this section.
        \begin{theorem}[Theorem \ref{main_theorem_DK_intro}]
            \label{main_theorem_DK_body}
            Consider a Grassmann flip $\mu:\Xc_1\dashrightarrow\Xc_2$ such that the exceptional divisor is isomorphic to a flag bundle with fiber $F(1,2,N)$ over a smooth projective base. Then $\dbcoh(\Xc_1)\subset\dbcoh(\Xc_2)$, i.e. $\mu$ satisfies the DK-conjecture.
        \end{theorem}
        \begin{proof}
            As for Theorem \ref{thm:main_theorem_fibrations_intro}, the proof for the even and odd case are essentially the same: we will only describe the even case (i.e. $\epsilon = 0$). In light of the decomposition \ref{eq:sod_from_left_flip}, the semiorthogonal complement $^\perp\pi_1^*\dbcoh(\Xc_1)$ can be described by the first chessboard of Section \ref{sec:chess_game_fibrations_even}, where the box in position $(x, y)$ containing a number $m$ now corresponds to the block $\mathbf T^m_{x, y}\subset\dbcoh(\Xc)$ By Lemma \ref{lem:chess_game_exts_flip}, all the chessboards of Section \ref{sec:chess_game_fibrations_even} describe the semiorthogonal complements of categories equivalent to $\pi_1^*\dbcoh(\Xc_1)\subset\dbcoh(\Xc)$, where the equivalence is described by a suitable mutation functor. In particular, the last chessboard corresponds to the following semiorthogonal decomposition:
            \begin{equation}\label{eq:collection_final_chesssboard_flip}
                \begin{split}
                    ^\perp\Xi\pi_1^*\dbcoh(\Xc_1) = \langle
                                            & \hspace{12pt}\mathbf T^{n-2}_{0,0}, \\
                                            & \hspace{12pt}\mathbf T^{n-1}_{0,1}, \mathbf T^0_{1,1}, \dots\dots\dots, \mathbf T^0_{n-1,1} \\
                                            & \times, \mathbf T^{n-1}_{0,2}, \mathbf T^0_{1,2}, \dots\dots\dots, \mathbf T^0_{n-1,2} \\
                                            & \hspace{20pt} \vdots \hspace{100pt} \vdots \\
                                            & \times, \mathbf T^{n-1}_{0,n}, \mathbf T^0_{1,n}, \dots\dots, \mathbf T^0_{n-1,n} \\
                                            & \times, \mathbf T^{n-2}_{0,n+1}, \mathbf T^0_{1,n+1}, \dots\dots, \mathbf T^0_{n-1,n+1} \\
                                            & \hspace{20pt} \vdots \hspace{100pt} \vdots \\
                                            & \times, \mathbf T^{n-2}_{0,2n-3}, \mathbf T^0_{1,2n-3}, \dots\dots, \mathbf T^0_{n-1,2n-3} \\
                                            & \hspace{12pt} \mathbf T^{n-2}_{0,2n-2}\rangle\\
                                            & \hspace{12pt} \mathbf T^{n-2}_{0,2n-1}\rangle,
                \end{split}
            \end{equation}
            where $\Xi$ is the mutation functor. Again, we observe that:
            \begin{equation*}
                \begin{split}
                    \sigma_* (p_1^* \dbcoh(\Gc r(2, \Ec))\otimes\Oc(0, 1))   & = \langle  \mathbf T^{n-1}_{0,1}, \dots, \mathbf T^{n-1}_{0,n}, \mathbf T^{n-2}_{0,n+1}, \dots, \mathbf T^{n-2}_{0,2n-2}, \mathbf T^{n-2}_{0,2n-1},\mathbf T^{n-2}_{0,2n}, \rangle \\
                                            & = \langle \mathbf T^{n-2}_{0,0}, \mathbf T^{n-1}_{0,1}, \dots, \mathbf T^{n-1}_{0,n}, \mathbf T^{n-2}_{0,n+1}, \dots, \mathbf T^{n-2}_{0,2n-2}, \mathbf T^{n-2}_{0,2n-1} \rangle.
                \end{split}
            \end{equation*}
            This allows to mutate the collection \ref{eq:collection_final_chesssboard_flip} so that the category of $\Xc$ can be written as:
            \begin{equation*}
                \dbcoh(\Xc) = \langle \sigma_* (p_1^* \dbcoh(\Gc r(2, \Ec))\otimes\Oc(0, 2)), \Zc, \Xi\pi_1^*\dbcoh(\Xc_1) \rangle.
            \end{equation*}
            where $\Zc$ is an admissible subcategory generated by (mutations of) blocks of the form $\mathbf T_{x,y}^m$. This concludes the proof.
        \end{proof}
\appendix
\section{Computations on homogeneous vector bundles}\label{app:borel_weil_bott}
\begin{lemma}\label{lem:symresolution}
    One has the following short exact sequence on $F(1,2,n)$:
    
    \begin{equation}\label{eq:symresolution}
        0\arw \Sym^{k-1}\Uc_{G(2, n)}^\vee(-1,1)\arw \Sym^k\Uc_{G(2, n)}^\vee\arw\Oc(k, 0)\arw 0
    \end{equation}
\end{lemma}

\begin{proof}
    For $k=1$ Equation \ref{eq:symresolution} reduces to the simple embedding of tautological bundles. To prove the assertion for higher $k$, we recall that the Schur functor $\Sym^k$ acts on a sequence $0\arw A\arw B\arw C\arw 0$ as follows:
    
    \begin{equation}\label{eq:wedgeresolution}
        0\arw \wedge^k A \arw \wedge^{k-1} A\otimes B \arw\cdots\arw\wedge^{k-l}A\otimes\Sym^l B\arw\cdots\arw\Sym^k B\arw\Sym^k C\arw 0.
    \end{equation}
    
    The proof follows by applying Equation \ref{eq:wedgeresolution} to  the embedding of tautological bundles: note that the first bundle has rank one, thus all its higher wedge powers are zero, giving the expected short exact sequence.
\end{proof}
The following corollary is immediate:
\begin{corollary}\label{lem:symresolution_flag_bundle}
    One has the following short exact sequence on $\Fc l(1,2,\Ec)$:
    \begin{equation}\label{eq:symresolution_flag}
        0\arw \Sym^{k-1}\Uc_2^\vee(-h_1+h_2)\arw \Sym^k\Uc_2^\vee\arw\Oc(kh_1)\arw 0
    \end{equation}
\end{corollary}
\begin{proof}
    It is enough to apply the functor $\Fm$ of Section \ref{sec:extending_bundles_from_fiber}.
\end{proof}
Another application is the following:
\begin{corollary}\label{cor:cohomology_of_sym_via_line_bundles}
    If $H^\bullet(\Fc l(1, 2, \Ec), \Oc((a-m+2k)h_1+(b+m-k)h_2)) = H^\bullet(\Fc l(1, 2, \Ec), \Oc((a-m)h_1+(b+m)h_2) = 0$, then $\Sym^m\wt\Uc^\vee(ah_1+bh_2)$ has no cohomology as well.
\end{corollary}
\begin{proof}
    The proof follows by applying Equation \ref{eq:symresolution_flag} iteratively, resolving all symmetric powers of $\Uc_2^\vee$ as extensions of progressively lower symmetric powers.
\end{proof}

\begin{lemma}\label{lem:main_coh_lemma}
    For $t\leq 2n-3+\epsilon$ one has:
            \begin{equation*}
                \Ext_{M}^\bullet(\Oc(t, 0), \Sym^r\Uc_{G(2, 2n+\epsilon)}^\vee(-1, 1)) \simeq \left\{
                \begin{array}{lr}
                    \CC[-1] & r = t-1 \\
                    0 & 0\leq r\leq \min(t-2, n-1), r\neq t-1
                \end{array}
                \right.
            \end{equation*}
    where $M$ is a $(1,1)$-section in $F(1,2, 2n+\epsilon)$.
\end{lemma}
\begin{proof}
    By the Koszul resolution of $M$, the computation boils down to
    \begin{equation*}
        \begin{split}
            \Ext_{F(1,2,2n+\epsilon)}^\bullet(\Oc(t, 0), \Sym^r\Uc_{G(2, 2n+\epsilon)}^\vee(-1,1)) \\
            \Ext_{F(1,2,2n+\epsilon}^\bullet(\Oc(t, 0), \Sym^r\Uc_{G(2, 2n+\epsilon)}^\vee(-2,0))
        \end{split}
    \end{equation*}    
    and hence to the cohomology of $\Sym^r\Uc_{G(2, 2n+\epsilon)}^\vee(-t-1, 1)$ and $\Sym^r\Uc_{G(2, 2n+\epsilon)}^\vee(-t-2, 0)$. In light of the iterate resolution of symmetric powers of Corollary \ref{cor:cohomology_of_sym_via_line_bundles}, we just need to prove that, for $0\leq \alpha\leq r$, the bundle $\Oc(-t-1-r+2\alpha, 1+r-\alpha)$ has no cohomology except for $\CC[-1]$ for $\alpha = r = t-1$, and $\Oc(-t-2-r+2\alpha, r-\alpha)$ has no cohomology at all. We apply the Borel--Weil--Bott theorem: the weight associated to the first bundle, once we add the sum of fundamental weights, can be expressed as:
    \begin{equation*}
        \omega + \rho = (-t-r+2\alpha, 2+r-\alpha, 1, \dots, 1)
    \end{equation*}
    where the number of $1$'s is $2n+\epsilon-3$. Observe that the first coordinate is always negative and the second one is always posytive or zero. If it is zero the bundle has no cohomology, otherwise we proceed by applying the Weyl reflection associated to the first simple root, finding:
    \begin{equation*}
        s_1(\omega + \rho) = (t+r-2\alpha, 2-t+\alpha, 1, \dots, 1).
    \end{equation*}
    The first coordinate is always positive, while the only way we can make the second coordinate positive as well is to choose $\alpha = r = t-1$, and this leads to a $\CC[-1]$ term in the cohomology. Let us assume now $\alpha\leq r\leq t-2$. Then, applying the Weyl reflection associated to the second simple root we find
    \begin{equation*}
        s_2\circ s_1(\omega + \rho) = (2+r-\alpha, -2+t-\alpha, 3-t+\alpha, 1, \dots, 1).
    \end{equation*}
    Observe that if $3-t+\alpha \leq 2n+\epsilon -4$, we eventually obtain a weight with all non negative coordinates and a zero coordinate, by simply repeating the step of applying the Weyl reflection which changes sign to the unique negative coordinate. The inequality we want can be rewritten as $t-\alpha \leq 2n+\epsilon -1$. However, by our assumptions we have $t-\alpha\leq t\leq 2n-3+\epsilon$.\\
    Let us now turn our attention to $\Oc(-t-2-r+2\alpha, r-\alpha)$. Its weight is:
    \begin{equation*}
        \omega + \rho = (-t-1-r+2\alpha, 1+r-\alpha, 1, \dots, 1)
    \end{equation*}
    where the number of $1$'s is $2n+\epsilon-3$. Again the first coordinate is always negative and the second one is always positive or zero. In the first case there is no cohomology, otherwise we apply the Weyl reflection associated to the first simple root:
    \begin{equation*}
        s_1(\omega + \rho) = (t+1+r-2\alpha, -t+\alpha, 1, \dots, 1).
    \end{equation*}
    We essentially proceed as above: there is no cohomology if we can prove that our assumptions imply $t-\alpha\leq 2n+\epsilon-3$, which is true because $t-\alpha\leq t\leq 2n-3+\epsilon$. This proves the claim.
\end{proof}
\begin{lemma}\label{lem:reordering_the_block}
    For $t\leq 2n-3+\epsilon$ and $0\leq r\leq \min(t-1, n-1)$ one has:
            \begin{equation*}
                \Ext_M^\bullet(\Oc(t, 0), \Sym^r\Uc_{G(2, 2n+\epsilon)}^\vee(-1,1)) \simeq 0.
            \end{equation*}
\end{lemma}
\begin{proof}
    The approach is the same as for the previous lemma, hence we will be brief. By Corollary \ref{cor:cohomology_of_sym_via_line_bundles}, we need to show that the bundles $\Oc(t+2\alpha, -r-\alpha)$ and $\Oc(-1+t+2\alpha, -1-r-\alpha)$ on $F(1,2,2n+\epsilon)$ have no cohomology for $0\leq \alpha\leq r$. If we add $\rho$ to the weight of the first one, we have
    \begin{equation*}
        \omega + \rho = (1+t+2\alpha, 1-r-\alpha, 1, \dots, 1)
    \end{equation*}
    where the first coordinate is positive and the second is negative. By applying the second Weyl reflection (i.e. the one associated to the second simple root) we get
    \begin{equation*}
        s_2(\omega + \rho) = (2-r+t+\alpha, -1+r+\alpha, 2-r-\alpha, 1, \dots, 1).
    \end{equation*}
    where the number of $1$'s is $2n+\epsilon-4$. As above, we see that there cannot be cohomology if $-2+r+\alpha\leq 2n+\epsilon-4$, i.e. $r+\alpha\leq 2n+\epsilon-2$. But by assumption $r+\alpha\leq 2r\leq 2n-2$.\\
    Consider now the second bundle: if we dd $\rho$ to its weight we have
    \begin{equation*}
        \omega + \rho = (t+2\alpha, -r-\alpha, 1, \dots, 1)
    \end{equation*}
    and as above:
    \begin{equation*}
        s_2(\omega + \rho) = (t-r+\alpha, r+\alpha, 1-r-\alpha, 1, \dots, 1)
    \end{equation*}
    where again the number of $1$'s is $2n+\epsilon-4$. Hence we need to prove that $-1+r+\alpha\leq 2n+\epsilon-3$, which holds by our assumptions.
    \begin{lemma}\label{lem:the_ext_for_the_second_part_of_rule_two}
        For $0\leq m\leq r\leq n-2$ the following vanishing holds:
        \begin{equation*}
            \Ext_M^\bullet(\Sym^m\Uc_{G(2, 2n+\epsilon)}^\vee, \Sym^r\Uc_{G(2, 2n+\epsilon)}^\vee) = 0.
        \end{equation*}
    \end{lemma}
    \begin{proof}
        As usual, we start with:
        \begin{equation*}
            \begin{split}
                \Ext_M^\bullet(\Sym^m\Uc_{G(2, 2n+\epsilon)}^\vee, \Sym^r\Uc_{G(2, 2n+\epsilon)}^\vee(-1, 1)) = \\
                = H^\bullet(M, \Sym^m\Uc_{G(2, 2n+\epsilon)}^\vee\otimes\Sym^r\Uc_{G(2, 2n+\epsilon)}^\vee(-1, 1-m))
            \end{split}
        \end{equation*}
        The bundle in the RHS is resolved by the following bundles on $F(1, 2, 2n+\epsilon)$:
        \begin{equation*}
            \begin{split}
                \Sym^m\Uc_{G(2, 2n+\epsilon)}^\vee\otimes\Sym^r\Uc_{G(2, 2n+\epsilon)}^\vee(-1, 1-m) \\
                \Sym^m\Uc_{G(2, 2n+\epsilon)}^\vee\otimes\Sym^r\Uc_{G(2, 2n+\epsilon)}^\vee(-2, -m)
            \end{split}
        \end{equation*}
        The first has no cohomology because its pushforward to $G(2, 2n+\epsilon)$, by projection formula, can be seen as a product of symmetric powers of $\Uc_{G(2, 2n+\epsilon)}^\vee$ times the pushforward of $\Oc(-1, 0)$, and the latter is identically zero. Let us now turn to the second bundle. By Serre duality, up to shifting by the dimension of $F(1, 2, 2n+\epsilon)$, computing its cohomology is equivalent to compute the cohomology of:
        \begin{equation*}
            \begin{split}
                \Sym^m\Uc_{G(2, 2n+\epsilon)}\otimes\Sym^r\Uc_{G(2, 2n+\epsilon)}(2, m)\otimes\omega_{F(1, 2, 2n+\epsilon)} \\
                = q^*\Hc om_{G(2, 2n+\epsilon)}(\Sym^m\Uc_{G(2, 2n+\epsilon)}^\vee(2n+\epsilon-1), \Sym^r\Uc_{G(2, 2n+\epsilon)}^\vee)
            \end{split}
        \end{equation*}
        Hence, if the latter has no cohomology, we are done, and this can be easily checked by considering the semiorthogonality conditions of the full exceptional collection for $G(2, 2n+\epsilon)$ of \cite{kuznetsov_isotropic_lines}.
    \end{proof}
\end{proof}

\bibliographystyle{alpha}
\bibliography{bibliography}

\end{document}